\documentclass[a4paper, 12pt]{article}

\usepackage[latin1]{inputenc} 
\usepackage[T1]{fontenc}      
\usepackage[all,knot, poly]{xy}
\usepackage{float}
\usepackage{graphicx}                  
\usepackage{amssymb}
\usepackage{pstricks,pst-plot,pst-text,pst-tree,pst-eps,pst-fill,pst-node,pst-math}
\usepackage{amsfonts}
\usepackage{setspace} 
\RequirePackage{amsmath}
\RequirePackage{amsthm}

\RequirePackage{latexsym}
\RequirePackage{amsfonts}
\RequirePackage{amsmath}
\RequirePackage{amssymb}
\RequirePackage{amsthm}

\newtheoremstyle{montheoreme}
{}
{}
{\itshape}
{}
{\bf}
{.}
{.5em}
{}

\newtheoremstyle{maremarque}
  {}
  {}
  {}
  {}
  {\sl}
  {.}
  {.5em}
  {}
\newtheoremstyle{manotation}
  {}
  {}
  {}
  {}
  {\bf}
  {.}
  {.5em}
  {}
\newtheoremstyle{monexo}
  {}
  {}
  {}
  {}
  {\sl}
  {}
  {.5em}
  {}
\usepackage{fancyhdr}
\theoremstyle{montheoreme}
\newtheorem{defn}{Definition}[section]
\newtheorem{thm}[defn]{Theorem}
\newtheorem{prop}[defn]{Proposition}
\newtheorem{lem}[defn]{Lemma}
\newtheorem{cor}[defn]{Corollary}

\theoremstyle{maremarque}
\newtheorem{rmq}[defn]{Remark}

\theoremstyle{manotation}
\newtheorem*{nota}{Notation}

\theoremstyle{monexo}

\newtheorem{exple}[defn]{Example}

\usepackage{sectsty}
\usepackage[footnotesize, margin=2.5cm]{caption}

\usepackage{pst-node}
\usepackage{pstricks,pst-plot,pst-text,pst-tree,pst-eps,pst-fill,pst-node,pst-math}

\definecolor{bordeaux}{rgb}{.545,0,0}

\title{Categorification of the Kazhdan-Lusztig basis of the Temperley-Lieb algebra by bimodules}

\author{Thomas Gobet}
\begin{document}
\maketitle
\begin{abstract}
We realize the Temperley-Lieb algebra by analogues of Soergel bimodules. The key point is that the monoidal structure is not given by a usual tensor product but by a slightly more complicated operation.

\end{abstract}
\tableofcontents
\section*{Introduction}
The purpose of this work was to try to realize the Temperley-Lieb algebra in type $A$ by bimodules, motivated by unexplained positivity properties in this algebra.

The category of Soergel bimodules defined in \cite{S} categorifies the Kazhdan-Lusztig basis of the Hecke algebra of any Coxeter system of finite rank. A diagrammatic categorification of the Temperley-Lieb category obtained by taking a quotient of (a diagrammatic version of) the category of Soergel bimodules in type $A$ was described in \cite{BE}. Such a quotient category is a priori not a category whose objects can be viewed as bimodules anymore, but Elias gives some indications that there could exist such a realization of it by considering quasi-coherent sheaves on \textit{Weyl lines}, that is, one dimensional subspaces of the geometric representation of the Coxeter group that are intersections of reflection hyperplanes. A natural framework for this is the analogues of Soergel bimodules that are suggested by Elias in \cite{BE}. As he noticed, such bimodules are not free anymore as left or right modules over the algebra of regular functions on the union of all Weyl lines.  

Writing $Z$ for the union of all the Weyl lines viewed as a subvariety of the geometric representation we are able to realize the Temperley-Lieb algebra as a monoidal category of graded $\bar{R}$-bimodules where $\bar{R}$ is the algebra of regular functions on $Z$ by considering a slightly more complicated operation than a usual tensor product: given two graded $\bar{R}$-bimodules $B, B'$, one can consider the right, resp. left annihilators of $B'$, resp. $B$ and associate to each of them the corresponding varieties $V_{B}^r$, $V_{B'}^\ell\subset Z$. We then define a product of bimodules by setting
$$B\ast B'=B\otimes_{\bar{R}}\mathcal{O}(V_{B}^r\cap V_{B'}^\ell)\otimes_{\bar{R}} B',$$ where $\mathcal{O}(-)$ stands for the algebra of regular functions. Unfortunately such a product is neither additive nor associative on the category of finitely generated graded $\bar{R}$-bimodules but it will be associative when restricted to a suitable stable class of bimodules containing some special bimodules called \textit{fully commutative} together with some of their sums and shifts; proving that the fully commutative bimodules are indecomposable, which is a long combinatorial argument, will allow us to extend our product to direct sums of shifts of fully commutative bimodules by bilinearity. Setting $B_i=\mathcal{O}(V_i)\otimes_{\mathcal{O}(V_i)^{s_i}} \mathcal{O}(V_i)$ for the analogue of the Soergel bimodule, where $V_i$ stands for the union of the Weyl lines not included in the reflecting hyperplane of $s_i$ where $s_i$ is the simple transposition $(i, i+1)$, we give a categorification theorem for the Temperley-Lieb algebra (theorem \ref{thm:fin}); the bimodule $B_i$ corresponds to the element $b_i$ of the Kazhdan-Lusztig basis of the Temperley-Lieb algebra and the $\ast$-product of bimodules to the multiplication in the Temperley-Lieb algebra.  

To be able to compute the $\ast$-product of bimodules $B_i$ we need to understand inductively their annihilators ; it turns out that given a bimodule $B_w=B_{i_1}\ast\cdots\ast B_{i_k}$ where $w=s_{i_1}\cdots s_{i_k}$ where one can pass from any reduced expression of $w$ to any other only by commutation relations (such elements of the Weyl group turn out to index the Kazhdan-Lusztig basis of the Temperley-Lieb algebra and are usually called \textit{fully commutative} or \textit{braid avoiding}), the left and right varieties corresponding to the left and right annihilators of $B_w$ can be characterized by two subsets of pairwise commuting reflections of the Weyl group, which turn out to be exactly the two sets obtained in the realization of the Temperley-Lieb algebra by planar diagrams by considering the diagram associated to the element $b_{i_1}\cdots b_{i_k}$ after removing the lines going from the top to the bottom of the diagram (that is, keeping only the half circles and viewing them as reflections by numbering the points from the left to the right). Hence this also gives some categorical interpretation of the realization of the Temperley-Lieb algebra by planar diagrams (proposition \ref{prop:digne}). \\
\\
Stroppel obtained in \cite{Strop} a categorification of the Temperley-Lieb algebra by considering projective functors on the principal block of graded parabolic versions of the BGG category $\mathcal{O}$; due to the relationship between Soergel bimodules and projective functors on category $\mathcal{O}$ (see \cite{Soergel}, Korollar $1$), we can expect our categorification to be related to the one obtained in \cite{Strop}.\\
\\
\textit{Organization of the paper}. Section $1$ gives some basic results on Weyl lines and introduces varieties and sets of reflections defined inductively, which will correpond to varieties associatied to left and right annihilators of bimodules. Section $2$ gives some results on graded bimodules as well as on analogues of Soergel bimodules considered here and introduces the product of bimodules; we show that when restricted to a suitable class of bimodules this product turns out to be associative. Section $3$ gives the realization of the Temperley-Lieb algebra by analogues of Soergel bimodules; for this we need to show the Temperley-Lieb relations and in order have a way to extend our product to direct sums of fully commutative bimodules we need to show that these are indecomposable.\\ 
\\
\textbf{Acknowledgement}. I thank my advisor François Digne at the University of Picardy, Amiens for his suggestions and for reading preliminary versions of this paper.  
\section{Combinatorics of Weyl lines}
The Coxeter systems $(\mathcal{W}, S)$ considered will always be of type $A$ unless otherwise specified, identifying $\mathcal{W}$ of type $A_n$ with the symmetric group on $n+1$ letters and $S$ with the set of simple transpositions $s_i=(i, i+1)$ for all $i=1,\dots, n$. We will denote by $T$ the set of reflections, by $H_t$ the reflecting hyperplane of $t\in T$ and by $V$ the geometric representation over a field $k$ of characteristic zero, which is reflection faithful in the sense of \cite{S}.
\subsection{Weyl lines}
\begin{defn}
A \emph{\textbf{Weyl line}} is a subspace of $V$ of dimension $1$ that is the intersection of reflection hyperplanes. A Weyl line is \textbf{transverse} to some reflection $t\in T$ if it is not contained in $H_t$. 
\end{defn}

We denote by $Z$ the union of all Weyl lines in $V$, which is a $\mathcal{W}$-stable subvariety of $V$. We write $V_t$ for the union of Weyl lines transverse to $t$ as a subvariety of $Z\subset V$ ; if the reflection is simple we will often write $V_i$ to mean $V_{s_i}$. 

\begin{lem}\label{lem:parabolic}
There exists a bijection
\begin{equation*}
\left\{\begin{array}{cccccc}\text{Weyl lines in $V$}\end{array}\right\} \overset{\sim}{\longrightarrow}  \left\{\begin{array}{cccccc} \text{partitions of $\{1,\dots,n+1\}$ into two subsets}  
\end{array}\right\},
\end{equation*}
which to any Weyl line $L=\bigcap_{i=1}^{n-1} H_{t_i}$, where $H_{t_i}$ is the reflection hyperplane of $t_i\in T$, associates the partition given by the decomposition of $t_1\cdots t_{n-1}$ into disjoint cycles (which turns out to be a partition in two sets as the proof will show).
\end{lem}
\begin{proof}
One has to show that the map defined above is well-defined. Suppose $L=\bigcap_{i=1}^{n-1} H_{t_i}$ is a Weyl line in  $V$. The product $w=t_1\cdots t_{n-1}$ has $T$-length equal to $n-1$ since $L$ has dimension $1$ (the set of roots of the $t_i$ consists of linearly independent vectors, which implies that $t_1\cdots t_{n-1}$ is a reduced $T$-decomposition for $w$ ; the parabolic subgroup generated by the $t_i$ is equal to the subgroup of elements of $\mathcal{W}$ fixing $L$; see \cite{Carter}, section $2$). Now the $T$-length of an element of the symmetric group $S_{n+1}$ is equal to $n+1$ minus the number of cycles occuring in the decomposition into disjoint cycles. This forces $w$ as element of $S_{n+1}$ to fix at most $1$ letter. If it fixes exactly one letter $j$, suppose $L$ is written as another intersection of reflecting hyperplanes $\bigcap_{i=1}^{n-1} H_{t_i'}$. Then all the $t_i'$ fix $L$ and hence have to be in the parabolic subgroup of $\mathcal{W}$ generated by the $t_i$. Hence all the $t_i'$ have to fix the letter $j$ and one gets the same partition of $n+1$ into two sets as before. 

If no letter is fixed, write $S_1\cup S_2$ for the disjoint union of the supports of the two cycles. If $L$ is written $\bigcap_{i=1}^{n-1} H_{t_i'}$, then every $t_j'$ has to be in the parabolic subgroup generated by the $t_i$ and since it is a conjugate of some $t_i$ it will either fix $S_1$ or fix $S_2$. Hence we obtain the same partition into two sets as before. 

Now for each partition $S_1\cup S_2$ of $\{1,\dots,n+1\}$ write a corresponding $n$-cycle if either $S_1$ or $S_2$ has cardinal one or write a corresponding product of $2$ cycles if both have cardinal more than $1$ and decompose them in the obvious way as products of $n-1$ reflections. This proves that the above map is surjective. Now if $L\neq L'$ are two different Weyl lines, then one can find some reflecting hyperplane $L'\subset H_s$ such that $L\cap H_s=0$. Then $s$ cannot be in the parabolic subgroup of elements fixing $L$ and hence $L$ and $L'$ will not yield the same cycle decomposition.

\end{proof}
\begin{rmq}
In fact, Weyl lines are in bijection with rank $n-1$ parabolic subgroups (that is, maximal parabolic subgroups, not necessarily standard).
\end{rmq}
\begin{lem}\label{lem:hyperebene}
Let $t, t', t''\in T$ be three distinct reflections not commuting with each other (in particular $t'tt'=tt't=t''$). Then 

$$V_t\cap V_{t'}=V_t\cap H_{t''}=V_{t'}\cap H_{t''}.$$

\noindent In particular $V_t\cap (V_{t'}\cup V_{t''})=V_t$.
\end{lem}
\begin{proof}
Let $L\subset V_t\cap V_{t'}$. By definition, $L=\bigcap_{i=1}^{n-1} H_{t_i}$ for some $t_i\in T$. Since $L$ is transverse to both $t$ and $t'$, $H_{r}\cap\bigcap_{i=1}^{n-1} H_{t_i}=0$ for $r=t, t'$. It follows that $t t_1\cdots t_{n-1}$ and $t' t_1\cdots t_{n-1}$ have reflection length equal to $n$ and hence that they are $(n+1)$-cycles. Since $t, t'$ are non-commuting, there exists distinct letters $i, k, k'\in\{1,\dots,n+1\}$ such that $t=(k,i)$, $t'=(k',i)$. An easy computation then shows that when considering the decomposition of $t_1\cdots t_{n-1}$ as a product of two cycles, the letters $k$ and $k'$ must lie in the same cycle and the letter $i$ must lie in the other cycle. This means that $tt't=(k,k')$ divides $t_1\cdots t_{n-1}$, which implies that $L\subset H_{tt't}=H_{t''}$. Conversely if $L\subset H_{t''}$ and $L\not\subset H_t$ then $L\not\subset H_{t'' t t''}=H_{t'}$ since $L$ is fixed by $t''$.
\end{proof}
\begin{rmq}
Identifying $W$ with the symmetric group and viewing a reflection as a transposition, if $t=(i, k)$ and $t'=(k, j)$ with $j\neq i$, then $V_t\cap V_{t'}$ consists exactly of the Weyl lines corresponding to the maximal parabolic subgroups whose operation on $\{1,\dots, n+1\}$ yields exactly two orbits $S_1$ and $S_2$ with $i, j\in S_1$ and $k\in S_2$. In particular $V_t\cap V_{t'}\neq\{0\}$.
\end{rmq}
\subsection{Noncrossing and dense sets of reflections}

\begin{nota} For $i\leq j$ two indices in $\{1,\dots,n\}$ we write $[i,j]$ for the set $\{i,i+1,\dots,j-1,j\}$.
\end{nota}
\begin{defn}
Two indices $i, j$ in $[1,n]$ are \textbf{distant} if $|i-j|>1$.
\end{defn}
To any sequence $i_1 \cdots i_m$ with $i_j\in\{1,\dots,n\}$ of length at least one, we associate the variety $W_{i_1\cdots i_m}$ built inductively by setting $W_i=V_i$ and 
$$W_{i_1\cdots i_m}=V_{i_1}\cap (W_{i_2\cdots i_m}\cup s_{i_1} W_{i_2\cdots i_m}).$$
These varieties will play a key role later on. We write $\mathcal{V}_n$ for the family of varieties obtained in this way.
\begin{exple}\label{ex:distant}
For $i$ and $j$ with $|j-i|>1$, one has $W_{ij}=V_i\cap V_j=W_{ji}$.
\end{exple}
\begin{exple}\label{ex:voisin}
We have $W_{i(i\pm1)}=V_i\cap(V_{i\pm1}\cup s_i V_{i\pm1})=V_i$ by lemma \ref{lem:hyperebene}.
\end{exple}
\begin{exple} One has $W_{i(i\pm 1)i}=V_i$.
\end{exple}
We will show in proposition \ref{prop:reflexion} that any $W\in\mathcal{V}_n$ can be written as an intersection $\bigcap_{t\in T_W} V_t$ for a unique set $T_W$ with interesting properties. 

\begin{lem}\label{lem:suitecroissante}
Let $j\leq m\leq i$. Then $W_{m(m-1)\cdots j}=V_m$ and $W_{m(m+1)\cdots i}=V_{m}$.
\end{lem}
\begin{proof}
We prove the first equality by induction on $m-j$, the second being similar. If $m-j=0$ then $W_m=V_m$ by definition. Suppose $m-j>0$. Then
$$W_{m(m-1)\cdots j}=V_m\cap(W_{(m-1)\cdots j}\cup s_m W_{(m-1)\cdots j})$$ and $W_{(m-1)\cdots j}=V_{m-1}$ by induction. Example \ref{ex:voisin} concludes.
\end{proof}
\begin{nota}
For short, if $s_i\in S$ is a simple reflection and $W\subset Z$ a closed subset, we write $s_i\cdot W$ or even $i\cdot W$ for the variety $V_i\cap(W\cup s_i W)$. More generally given any sequence $i_1\cdots i_k$ of indices in $\{1,\dots,n\}$, we write $i_1\cdots i_k\cdot W$ for the variety $s_{i_1}\cdot(s_{i_2}\cdot(\cdots(s_{i_k}\cdot W)\cdots )$.
\end{nota}

\begin{lem}\label{lem:reflexions-recurrence}
Suppose $Q\subset T$ is a set of commuting reflections. Let $s\in T$. Set $W:=\bigcap_{t\in Q} V_t$. Then $W\neq 0$ and $s\cdot W=\bigcap_{t\in Q'} V_t$ where
$$
Q' = \left\{
    \begin{array}{ll}
        Q\cup\{s\} & \mbox{if } st=ts \mbox{ for each } t\in Q\\
        (Q\backslash t)\cup\{s\} & \mbox{if } \exists! t\in Q \mbox{ such that } st\neq ts\\
        (Q\backslash\{t, t'\})\cup\{s, tt'st't\} & \mbox{if } \exists t\neq t'\in Q \mbox{ such that } st\neq ts, st'\neq t's.
    \end{array}
\right.
$$ 
and $Q'$ is also commuting. In particular $s\cdot W\neq 0$.
\end{lem}
\begin{proof}
First notice that $W\neq \{0\}$: since the reflections from $Q$ pairwise commute, any Weyl line corresponding to a parabolic subgroup $P$ with the following property will be in $W$: the operation of $P$ on $\{1,\dots,n+1\}$ yields two orbits $S_1$ and $S_2$ where each $t\in Q$ has an index from its support in $S_1$ and the other one in $S_2$. The same will hold for $s\cdot W$. The fact that the sets $Q'$ are commuting is obvious in the two first cases ; for the third case it is an easy computation viewing the reflections as transpositions and considering their supports. 

First recall that for $s, t$ any two reflections, $sV_t=V_{sts}$. If $s\in T$ commutes with any of the $t\in Q$ then $s\cdot W=\left(\bigcap_{t\in Q} V_t\right)\cap V_s$ since $s V_t=V_{sts}=V_t$ whenever $s$ and $t$ commute. 

If $s t\neq t s$ for some $t\in T_W$ but $s$ commutes with any $t'\in Q$ with $t'\neq t$, then 
$$s\cdot W=\left(\bigcap_{r\in Q\backslash t} V_{r}\right)\cap V_s\cap( V_{t}\cup V_{s t s})$$
\noindent As we have seen in lemma \ref{lem:hyperebene} we have $V_s\cap( V_{t}\cup V_{s t s})=V_s$ hence
$$s\cdot W=\bigcap_{r\in (Q\backslash t)\cup\{s\}} V_{r},$$
The remaining case is the case where $s$ does not commute with exactly two reflections $t, t'\in Q$. In that case one has
$$s\cdot W=\left(\bigcap_{r\in Q\backslash\{t,t'\}} V_r\right)\cap V_s\cap((V_{t}\cap V_{t'})\cup (V_{sts}\cap V_{st's})).$$
We claim that 
$$V_s\cap((V_{t}\cap V_{t'})\cup (V_{sts}\cap V_{st's}))=V_s\cap V_{t s t' s t},$$
which concludes. By lemma \ref{lem:hyperebene} we have
\begin{eqnarray*}
V_s\cap V_{t}\cap V_{t'}&=&V_t\cap V_{t'}\cap H_{st's}=V_{t'}\cap(V_t\cap H_{st's})\\
&=&V_{t'}\cap V_t\cap V_{tst'st}.
\end{eqnarray*}
Similarly, $V_s\cap V_{s t s}\cap V_{s t' s}=V_{sts}\cap V_{st's}\cap V_{tst'st}$.
Conversely, since $V_s\cap V_{t s t' s t}$ in not equal to zero consider a Weyl line $L\subset V_s\cap V_{t s t' s t}$. If $L\not\subset H_{t}$, then $L\subset H_{s t' s}$ and hence $L\not\subset H_{t'}$ since $L\not\subset H_{s}$. Similarly if $L\in V_s\cap V_{t s t' s t}$ and $L\subset H_{t}$ then $L\not\subset H_{s t' s}$ (since $L\not\subset H_{t s t' s t}$) and $L\not\subset H_{s t s}$ (since $L\not\subset H_{s}$).
\end{proof}
\begin{prop}\label{prop:reflexion}
Let $W\in\mathcal{V}_n$. Then $W\neq \{0\}$ and there exists a unique set $T_W\subset T$ with $tt'=t't$ for each $t, t'\in T_W$ such that $W=\bigcap_{t\in T_W} V_{t}$.
\end{prop}
\begin{proof}
Existence is shown using induction on the length of a sequence associated to a variety in $\mathcal{V}_n$. If $W\in\mathcal{V}_n$ is obtained from a sequence of length $1$, then $W=V_j$ for some $j$ and $W\neq 0$. Now assume the result holds for each variety in $\mathcal{V}_n$ obtained from a sequence of length less than or equal to $m$, and suppose $W\in\mathcal{V}_n$ is obtained from a sequence of length equal to $m+1$. Then by definition $W=s\cdot W'$ for some simple reflection $s$ and some variety $W'\in\mathcal{V}_n$ obtained from a sequence of length equal to $m$. By induction $W'=\bigcap_{t\in T_{W'}} V_{t}$ and thanks to lemma \ref{lem:reflexions-recurrence} we have $W=\bigcap_{t\in Q'} V_t$ with $Q'$ commuting and $W\neq \{0\}$.

For unicity, suppose $W\in\mathcal{V}_n$ and suppose there exists another set $Q$ of pairwise commuting reflections such that $W=\bigcap_{t\in Q} V_t$. Let $s\in Q$. Suppose $s\notin T_W$. If there exists $t\in T_W$ such that $ts\neq st$, then $W\subset V_t\cap V_s\subset H_{sts}$ using lemma \ref{lem:hyperebene}. But this is impossible because since $t$ commutes to any reflection in $T_W$, $W$ is $t$-invariant, hence $W=t W\subset t V_s=V_{sts}$. Now suppose $s$ commutes with any reflection in $T_W$. In type $A$, a set of commuting reflections contains at most $\left\lfloor \frac{n+1}{2} \right\rfloor$ elements. Hence $|T_W|+1\leq \left\lfloor \frac{n+1}{2} \right\rfloor$. First suppose $|T_W|+1< \frac{n+1}{2}$. As a consequence, if $n>1$, there must exist a reflection $s'\in T$ such that $s'$ commutes with any element of $T_W$ but not with $s$ (think about identifying reflections with transpositions and considering their supports). If $L\subset W$ is a Weyl line, then by assumption $L\subset V_s$. This forces $L\not\subset H_{s'}$ or $L\not\subset H_{ss's}$ : otherwise $L=s'L\subset H_s$. Suppose $L\not\subset H_{s'}$, which implies by lemma \ref{lem:hyperebene} that $L\subset H_{ss's}$ and $s' L\subset H_s$. But since $s'$ commutes with any reflection in $T_W$ one has that $W$ is $s'$-stable, hence $s'L\subset W\subset V_s$, a contradiction. The case where $L\subset H_{s'}$ is similar, permuting $s'$ and $ss's$. Hence any case with $s\notin T_W$ leads to a contradiction. This forces $Q\subset T_W$ and also $T_W\subset Q$ by exchanging the roles of $Q$ and $T_W$.

Now suppose $|T_W|+1=\frac{n+1}{2}$. Write $T_W\cup\{s\}=\{t_1,\dots, t_{k}, s\}$ and we can suppose without loss of generality that $t_i=s_{2i-1}$, $s=s_{2k+1}$. Notice that $k=(n-1)/2$. Consider the intersection of hyperplanes 
$$H_{s_1 s_2 s_1}\cap H_{s_2 s_3 s_2}\cap H_{s_3 s_4 s_3}\cap\cdots\cap H_{s_{2k-2} s_{2k-1} s_{2k-2}}\cap H_{s_{2k-1} s_{2k} s_{2k-1}}\cap H_{s_{2k+1}}$$
\noindent that involves $2k=n-1$ reflecting hyperplanes. The product 
$$(s_1 s_2 s_1)(s_2 s_3 s_2)\cdots(s_{2k-2} s_{2k-1} s_{2k-2})(s_{2k-1} s_{2k} s_{2k-1}) s_{2k+1}$$
has reflection length equal to $n-1$. It follows that the above intersection of hyperplanes is a Weyl line $L$. But then $L\subset H_s$ and $L\not\subset H_t$ for each $t\in T_{W}$ : if $L\subset H_{s_j}$ for some $j=1,3,\dots, 2k-1$, it follows by successive conjugations that $L\subset H_{s_j}$ for each index $j=1,\dots, n$ and hence that $L=0$. Hence $\bigcap_{t\in T_W} V_t\neq \bigcap_{t\in T_W\cup\{s\}} V_t$.   
\end{proof}
\begin{rmq}\label{rmq:reflection}
When proving unicity in the above proof we have shown that if $W\in\mathcal{V}_n$ and $W\subset V_t$ for some reflection $t\in T$, then $t\in T_W$ and in particular $W$ is $t$-invariant. Hence we have :
\end{rmq}
\begin{prop}
Let $W\in \mathcal{V}_n$. Then
$$T_W=\{s\in T~|~W\subset V_s\}.$$
\end{prop}
\noindent The following consequence will be crucial further:
\begin{cor}\label{cor:ni}
Let $W\in\mathcal{V}_n$, $i\in\{1,\dots,n\}$. Then $V_i\cap W\neq\{0\}$ and the following are equivalent:
\begin{enumerate}
\item The variety $W$ is $s_i$-invariant,
\item The variety $W\cap V_i$ is $s_i$-invariant,
\item For each $t\in T$ such that $ts_i\neq s_it$, $(W\cap V_i)\cap V_t\neq 0$. 
\end{enumerate}
\end{cor}
\begin{proof}
Thanks to the above proposition, $W=\bigcap_{t\in T_W} V_t$, where $T_W\subset T$ is a set of pairwise commuting reflections. Hence we can find a partition $S_1\cup S_2=\{1,\dots,n\}$ such that $i\in S_1$, $i+1\in S_2$ and each $t\in T_W$ can be written as a transposition $(j,k)$ with $j\in S_1$ and $k\in S_2$. Thanks to lemma \ref{lem:parabolic} this gives us a corresponding Weyl line included in $W\cap V_i$, hence $W\cap V_i\neq\{0\}$. If $W$ is $s_i$-invariant then so is $W\cap V_i$, and then if $W\cap V_i\subset H_t$ for some reflection $t$ which does not commute with $s_i$, one would get $W\cap V_i\subset H_{s_its_i}$ by $s_i$-invariance and hence also $W\cap V_i\subset H_i$ which would force $W\cap V_i=\{0\}$. Now if $W$ is not $s_i$-invariant, there exists $t'\in T_W$ such that $t' s_i\neq s_i t'$ and $V_i\cap W\subset V_i\cap V_{t'}\subset H_{s_i t' s_i}$ by lemma \ref{lem:hyperebene}, and $t=s_i t' s_i$ does not commute with $s_i$ since $t'$ does not. 
\end{proof}
\begin{defn}
A set $Q\subset T$ of pairwise commuting reflections is \textbf{noncrossing} if after identification with a set of transpositions of the isomorphic symmetric group, it contains no pair of transpositions $(i, j)$ and $(k, l)$ with $i<k<j<l$. 
\end{defn}
If we draw $n+1$ points on a circle and label each of them with an index between $1$ and $n+1$, starting by $1$ at some point and writing the increasing indices in clockwise order, and represent a transposition by a line segment between the two indices it exchanges, a set $Q\subset W$ of reflections is noncrossing if and only if any two segments in the corresponding circle never cross each other. Equivalently, if one draws a line with $n+1$ points starting on the left by $1$ and represent a transposition by an arc between the two indices it exchanges (up to homotopy), then a set of reflections is noncrossing if and only if there is a way of writing the arcs such that any two arcs associated to distinct reflections from this set never cross. This last way of representing noncrossing sets will turn out to be the most convenient one.  

\begin{defn}
Given any subset $Q\subset \mathcal{W}$, the \textbf{support} of $Q$, written $\mathrm{supp}(Q)$, is the union of the supports of its elements viewed as elements of the symmetric group. A set $Q\subset T$ of pairwise commuting reflections will be said to be \textbf{dense} if it is noncrossing and if there exists an integer $k>0$ and indices $0<m_1<j_1<m_2<j_2<\dots <m_k<j_k\leq n+1$ such that $(m_q, j_q)\in Q$ and $\mathrm{supp}(Q)=\bigcup_{q=1}^k \{m_q, m_q+1,\dots, j_q\}$. This forces in particular $j_q-m_q$ to be odd for each $q$ since $Q$ is noncrossing and $(m_q,j_q)\in Q$. A subset of $\mathrm{supp}(Q)$ of the form $\{m_q, m_q+1,\dots, j_q\}$ as above will be called a \textbf{block} of indices from $Q$. 
\end{defn}
\begin{lem}
Let $W\in\mathcal{V}_n$. Then $T_W$ is noncrossing.
\end{lem}
\begin{proof}
Again, we use induction on the length of the sequence defining $W$. If such a sequence has length one the result is clear. Let $W=s\cdot W'$ and suppose $Q=T_{W'}$ is noncrossing, then $Q'=T_{s\cdot W'}$ is also noncrossing using the formulas from lemma \ref{lem:reflexions-recurrence} (it is obvious in the two first cases and clear for the last one if we represent $Q$ and $Q'$ as arcs joining points on a line).

\end{proof}
\begin{nota}
If $W\in\mathcal{V}_n$ is associated to a sequence $i_1\cdots i_k$ we will often write $T(i_1\cdots i_k)$ instead of $T_W$ for convenience. Notice that using lemma \ref{lem:reflexions-recurrence} one can inductively compute the variety and the corresponding dense set associated to a sequence.  
\end{nota}
\begin{thm}\label{thm:classification}
Let $W\in \mathcal{V}_n$. Then $T_W$ is dense. Conversely, any dense subset $Q\subset T$ is equal to a $T_{V'}$ for some variety $V'\in\mathcal{V}_n$. In formulas, 
$$\{T_W~|~W\in\mathcal{V}_n\}=\{Q\subset T~|~Q~\text{is dense}\}.$$
\end{thm} 
\begin{proof}
Thanks to the previous lemma $T_W$ is noncrossing for each $W\in\mathcal{V}_n$. If $W$ is associated to a sequence of length $1$ then $T_W$ contains only one simple reflection, hence is dense. It suffices then to show that the rules from lemma \ref{lem:reflexions-recurrence} preserve dense sets, which is clear for the first two rules and easy for the last one if we write the reflections as transpositions.  

Conversely suppose that $Q$ is dense, in particular $\mathrm{supp}({Q})=\bigcup_{q=1}^{k} [m_q, j_q]$, with $j_q-m_q$ odd for each $1\leq q\leq k$. Consider the set of simple reflections $\bigcup_{q=1}^k \{s_{m_q}, s_{m_q+2},\dots,s_{j_q-1}\}$ and rewrite this union as $\{s_{k_1},\dots, s_{k_{n(Q)}}\}$ with $k_i<k_j$ if $i<j$. Notice that this is a set of pairwise commuting reflections. We will show by induction on the size of the biggest block of $Q$ that there exists a sequence $\mathrm{seq}=n_1 n_2\cdots n_\ell$ with $n_i\in\bigcup_{q=1}^{k} [m_q,j_q-1]$ for each $1\leq i\leq\ell$ such that $Q=T_W$ where $W$ is associated to the sequence
$$\mathrm{seq} k_1 k_2 \cdots k_{n(Q)}$$
obtained by concatenation of the sequence $\mathrm{seq}$ and the sequences $k_1 k_2\cdots k_{n(Q)}$. First we suppose that the size of the biggest block is $1$. Then each block has size one, in other words, $j_q=m_q+1$ for each $q$ and there is only one corresponding dense set $Q$: the set of reflections $\{s_{k_1}, s_{k_2}, \dots, s_{k_{n(Q)}}\}$. One then has $Q=T_W$ with $W$ associated to the sequence $k_1 k_2\cdots k_{n(Q)}$ (see example \ref{ex:distant}).
Now suppose that the biggest block $\mathcal{B}_i=[m_i, j_i]$ of $Q$ has size bigger than $1$. It suffices to show the induction hypothesis for the set $Q_i$ of reflections in $Q$ supported in $\mathcal{B}_i$, i.e., that $Q_i$ is equal to $T_W$ for some $W$ associated to a sequence $\mathrm{s(i)}=\mathrm{seq}_i m_i (m_i+2)\cdots (j_i-1)$ where $\mathrm{seq}_i$ is a sequence with all indices in $[m_i, j_i-1]$ : if this holds, one associates to each block $\mathcal{B}_q$ of $Q$ the variety $W_{\mathrm{s(q)}}$ such that $T_{W_{\mathrm{s(q)}}}$ is equal to the set $Q_q$ of reflections in $Q$ supported in $\mathcal{B}_q$ (this is possible since we show it for the biggest block(s) and the result holds by induction for blocks of smaller size); but then if $q\neq q'$ the reflections in $Q_q$ commute with the reflections in $Q_{q'}$ since they are supported in $[m_q, j_{q}]$ and $[m_{q'}, j_{q'}]$ which are disjoint. Hence one gets
\begin{eqnarray*} 
Q&=&\bigcup_{q=1}^k T({\mathrm{s}(q)})\\
&=&T({\mathrm{s}(1)\cdots \mathrm{s}(k)})\\
&=&T(\mathrm{seq}_{1}m_1(m_1+2)\cdots(j_1-1)\cdots \mathrm{seq}_{k} m_k(m_k+2))\cdots (j_k-1)\\
&=&T(\mathrm{seq}_{1}\cdots \mathrm{seq}_{k}\underbrace{m_1(m_1+2)\cdots(j_1-1)\cdots m_k(m_k+2)\cdots (j_k-1)}_{=k_1k_2\cdots k_{n(Q)}}),
\end{eqnarray*}
where second and last equalities hold since the indices in $\mathrm{s}(i)$ are distant from the indices $\mathrm{s}(i')$ whenever $i\neq i'$ (if two sequences $\mathrm{x}$ and $\mathrm{y}$ are such that any index in $\mathrm{x}$ is distant from any index in $\mathrm{y}$ then it is a consequence of lemma \ref{lem:reflexions-recurrence} that $T(\mathrm{xy})=T(\mathrm{yx})=T(\mathrm{x})\cup T(\mathrm{y})$).

Therefore we have to show that a dense set $Q$ having only one block $[k_1,k_{n(Q)}+1]$ can be obtained as $T_W$ for $W$ associated to a sequence obtained by concatenating a sequence with indices in $[k_1,k_{n(Q)}]$ to the left of $k_1\cdots k_{n(Q)}$ ; since $Q$ has a single block we have $k_ {j+1}=k_j+2$ for each $k=1,\dots, n(Q)-1$. We first show that we can concatenate a sequence to the left of this sequence to obtain a corresponding variety $W'$ such that $T_{W'}=Q'$ contains exactly the reflection $(k_1, k_{n(Q)}+1)$ and all the simple reflections $(k_1+1, k_2), (k_2+1, k_3),\dots, (k_{n(Q)-1}+1,k_{n(Q)})$ and then we will build $W$ from $W'$ by induction ; see figure \ref{figure:proc} for an illustration of this process. By induction using lemma \ref{lem:reflexions-recurrence} we get that $T_{W_{(k_i+1)\cdots(k_{n(Q)-1}+1)k_1\cdots k_{n(Q)}}}$ is equal to the set
$$\{s_{k_1},s_{k_2},\dots, s_{k_{i-1}}, (k_i, k_{n(Q)}+1), s_{k_i+1}, s_{k_{i+1}+1}, \dots, s_{k_{n(Q)-1}+1}\},$$
hence $Q'=T_{W'}$ where $W'$ is associated to the sequence
$$(k_1+1)(k_2+1)\cdots(k_{n(Q)-1}+1)k_1\cdots k_{n(Q)}.$$

Now if we remove the reflection $(k_1, k_{n(Q)}+1)$ from $Q$ we obtain again a dense subset $Q''$ with support equal to $[k_1+1, k_{n(Q)}]$ by density of $Q$ ; hence all blocks of $Q''$ have a size smaller than $k_{n(Q)}+1-k_1$. Hence by induction, $Q''$ can be obtained as a $T_{W''}$ for $W''$ associated to a sequence $\mathrm{seq} (k_1+1) (k_2+1)\cdots (k_{n(Q)-1}+1)$ for some sequence $\mathrm{seq}$ having all its indices lying in $\{k_1+1,\dots, k_{n(Q)-1}+1\}$. But then $s=(k_1, k_{n(Q)}+1)$ commutes with any reflection $s_\ell$ where $\ell$ is an index in $\mathrm{seq}$, hence one has
\begin{eqnarray*}
W_{\mathrm{seq} (k_1+1) \cdots (k_{n(Q)-1}+1) k_1\cdots k_{n(Q)}}&=&\mathrm{seq}\cdot W_{(k_1+1)\cdots (k_{n(Q)-1}+1) k_1\cdots k_{n(Q)}}=\mathrm{seq}\cdot W'\\&=&\mathrm{seq}\cdot\left(\bigcap_{t\in Q'} V_t\right)\\
&=&\mathrm{seq}\cdot\left( V_s\cap\bigcap_{i=1}^{n(Q)-1} V_{s_{k_i+1}}\right)\\&=&\mathrm{seq}\cdot(V_s\cap W_{(k_1+1)\cdots(k_{n(Q)-1}+1)})\\&=&V_s\cap (\mathrm{seq}\cdot W_{(k_1+1)\cdots(k_{n(Q)-1}+1)})\\
&=&V_s\cap\bigcap_{t\in Q''} V_t=\bigcap_{t\in Q} V_t,
\end{eqnarray*}
and the sequence $\mathrm{seq} (k_1+1) (k_2+1)\cdots (k_{n(Q)-1}+1)$ has all its indices lying in $[k_1,k_{n(Q)-1}+1]\subset [k_1,k_{n(Q)}]$.
\end{proof}
\begin{figure}[h!]
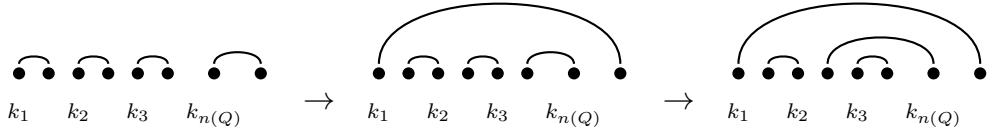

\begin{center}
\begin{tabular}{ccccc}
$~$ & & & &\\
\begin{psmatrix}[colsep=0.24,rowsep=0.4]
\pscircle*{0.08} & \pscircle*{0.08} & \pscircle*{0.08} & \pscircle*{0.08} & 
\pscircle*{0.08} & \pscircle*{0.08} & \pscircle*{0.08} & \pscircle*{0.08}\\
\scriptsize{$k_1$}&&\scriptsize{$k_2$}&&\scriptsize{$k_3$}&&\scriptsize{$k_{n(Q)}$}&
\end{psmatrix}
\ncarc[arcangle=90]{1,1}{1,2}
\ncarc[arcangle=90]{1,3}{1,4}
\ncarc[arcangle=90]{1,5}{1,6}
\ncarc[arcangle=90]{1,7}{1,8}
& $\rightarrow$ &
\begin{psmatrix}[colsep=0.24,rowsep=0.4]
\pscircle*{0.08} & \pscircle*{0.08} & \pscircle*{0.08} & \pscircle*{0.08} & 
\pscircle*{0.08} & \pscircle*{0.08} & \pscircle*{0.08} & \pscircle*{0.08}\\
\scriptsize{$k_1$}&&\scriptsize{$k_2$}&&\scriptsize{$k_3$}&&\scriptsize{$k_{n(Q)}$}&
\end{psmatrix}
\ncarc[arcangle=90]{1,1}{1,8}
\ncarc[arcangle=90]{1,2}{1,3}
\ncarc[arcangle=90]{1,4}{1,5}
\ncarc[arcangle=90]{1,6}{1,7}
& $\rightarrow$ &
\begin{psmatrix}[colsep=0.24,rowsep=0.4]
\pscircle*{0.08} & \pscircle*{0.08} & \pscircle*{0.08} & \pscircle*{0.08} & 
\pscircle*{0.08} & \pscircle*{0.08} & \pscircle*{0.08} & \pscircle*{0.08}\\
\scriptsize{$k_1$}&&\scriptsize{$k_2$}&&\scriptsize{$k_3$}&&\scriptsize{$k_{n(Q)}$}&
\end{psmatrix}
\ncarc[arcangle=90]{1,1}{1,8}
\ncarc[arcangle=90]{1,2}{1,3}
\ncarc[arcangle=90]{1,4}{1,7}
\ncarc[arcangle=90]{1,5}{1,6}\\
$~$ & & & &
\end{tabular}
\end{center}
\caption{Illustration of the process used in the proof of theorem \ref{thm:classification} to build a block $Q$ of maximal size from the sequence $k_1 \cdots k_{n(Q)}$ with $n(Q)=4$. On the left is the dense subset associated to this sequence; in the middle is the block $Q'$ associated to the sequence $(k_1+1)\cdots(k_{n(Q)-1}+1)k_1\cdots k_{n(Q)}$; on the right is the block $Q$. The dense set $Q''$ is obtained from $Q$ by removing the reflection represented by the arc joining $k_1$ to $k_{n(Q)}+1$.} 
\label{figure:proc}
\end{figure}
\section{Quasi-coherent sheaves on Weyl lines}

\subsection{Regular functions}
Let $R$ be the algebra of regular functions on $V$ and $\bar{R}$ be the algebra of regular functions on $Z$. Notice that $R\twoheadrightarrow \bar{R}$. For each subset $J\subset T$, we write $R_J$ for the algebra of regular functions on the union of Weyl lines transverse to any element in $J$. If the reflection considered are simple, we will write $R_i$ instead of $R_{\{s_i\}}$, $R_{i,j}$ instead of $R_{\{s_i,s_j\}}$, etc.

We denote by $f_k$ an element of $R$ or $\bar{R}$ which is an equation of the reflecting hyperplane $H_{s_k}$. We will often abuse notation and write $f_i$ for $f_i|_X$ where $X$ is a subvariety of $Z$.

If $X\subset V$ is a Zariski closed subset which is $t$-stable for $t\in T$, then $t$ induces a map $\mathcal{O}(X)\rightarrow\mathcal{O}(X)$ and one has a decomposition into eigenspaces $\mathcal{O}(X)=\mathcal{O}(X)^t\oplus \mathcal{O}(X)^t f_t$ where $f_t$ is an equation of the reflecting hyperplane $H_t$. If moreover no irreducible component of $X$ lies in $H_t$, then the Demazure operator $\partial_t : R\rightarrow R, f\mapsto (2f_t)^{-1}(f-tf)$ induces a map $\mathcal{O}(X)\rightarrow\mathcal{O}(X)$ and as $R^t$-modules $\mathcal{O}(X)^t\xrightarrow{\sim} \mathcal{O}(X)^t f_t$ where the isomorphism is given by multiplication by $f_t$ and its inverse by the restriction of $\partial_t$. 

\begin{rmq}\label{rmq:noninvariant}
A consequence of corollary \ref{cor:ni} which will be crucial further is the following : suppose $W\cap V_i$ is not $s_i$-invariant. Then $W\cap V_i\subset H_t$ for some $t\in T$ such that $t s_i\neq s_i t$. Then $t=(i, k)$ or $(i+1, k)$ for some $k\neq i,i+1$, say $t=(i+1, k)$. Suppose $k<i$. In $H_t$ one has
$$f_k+f_{k+1}+\dots+ f_{i}=0,$$
\noindent hence
$$f_i=-2 f_k-\cdots -2 f_{i-1} -f_i.$$
Viewing the right hand side in $R_i$ one sees that it lies in $R_i^{s_i}$. One can do the same for the other cases (the case where $k>i+1$ and the cases where $t=(i, k)$). Since $R_i=R_i^{s_i}\oplus R_i^{s_i} f_i$ one has that $R_i^{s_i}\twoheadrightarrow \mathcal{O}(W\cap V_i)$. In other words, when choosing a function $f$ in $R_i$ such that $f|_{W\cap V_i}$ is equal to a given $g\in \mathcal{O}(W\cap V_i)$, one can always suppose $f$ is $s_i$-invariant. 
\end{rmq}

\subsection{Graduations}

The Temperley-Lieb algebra will be realized via $(\bar{R},\bar{R})$-bimodules. Now in order to interpret the parameter in a categorification of the Temperley-Lieb algebra, the bimodules we will consider need to be $\mathbb{Z}$-graded. If $A, B$ are two $\mathbb{Z}$-graded rings, we write $A-\mathrm{mod}-B$ for the category of $A\otimes B^{\mathrm{op}}$-modules (that we will call "$(A, B)$-bimodules") and $A-\mathrm{mod}_\mathbb{Z}-B$ for the category of $\mathbb{Z}$-graded $A\otimes B^{\mathrm{op}}$-modules (that we will call "graded $(A,B)$-bimodules") with morphisms the bimodule morphisms that are homogeneous of degree zero. In all the cases we will consider in this document, $A$ and $B$ will be commutative rings, hence both operations give left or right-module structures. However, to distinguish the operations for example in case $A=B$, we will always refer to the operation of $A$ as the "left" operation and the operation of $B$ as the "right" operation on a $(A,B)$-bimodule $M$.   

\begin{nota}
If $M\in A-\mathrm{mod}_{\mathbb{Z}}-B$, we write $M[k]$ for the bimodule equal to $M$ in $A-\mathrm{mod}-B$ but with graduation shifted by $k$, that is, $(M[k])_i=M_{i+k}$.
\end{nota}

The algebra $R$ of regular functions on $V$ is naturally graded ; we use the convention that it is positively graded with $R_2=V^*$. Now $I(Z)$ is the intersection of the ideals of all the Weyl lines and the ideal of a line is homogeneous ; hence $I(Z)$ is also homogeneous and $\bar{R}$ inherits a $\mathbb{Z}$-grading from $R$. From now on the word "graded" will always mean "$\mathbb{Z}$-graded". 

\begin{lem}\label{lem:stroppel}
Let $A, B, C$ be graded rings, let $M\in A-\mathrm{mod}_{\mathbb{Z}}-B$ and $N\in B-\mathrm{mod}_{\mathbb{Z}}-C$. Then $M\otimes_B N$ lies in $A-\mathrm{mod}_{\mathbb{Z}}-C$.
\end{lem}
\begin{proof}
See \cite{Stroppel} lemma $1.2$, where $N$ has only a left-module structure : the graded decomposition $B=\bigoplus B_i$ of the tensor product as left module which is built in the proof of this lemma is also a graded decomposition in case we have an additional right-module structure on $N$ and hence on the tensor product, so the same proof can be given in our case. 
\end{proof}

\begin{lem}\label{lem:annulateur}
Let $A, B, C$ be (graded) rings, $f:C\rightarrow A$ a morphism of (graded) rings, $\pi:A\twoheadrightarrow A'$ a quotient of $A$ (by an homogeneous ideal), $M$ a (graded) module in 
$B-\mathrm{mod}-C$. Let $I\subset A$ be an (homogeneous) ideal which annihilates $M\otimes_C A'$ on the right. Then one has an isomorphism in $B-\mathrm{mod}-C$, resp. $B-\mathrm{mod}_{\mathbb{Z}}-C$ 
$$M\otimes_C A'\cong M\otimes_C (A'/\pi(I)).$$
\end{lem}
\begin{proof} Write $\psi:A'\twoheadrightarrow A'/\pi(I)$ for the canonical surjection and define a map $\varphi: M\otimes_C A'\rightarrow M\otimes_C (A'/\pi(I))$ by $\varphi(m\otimes n)=m\otimes\psi(n)$. It is well defined and defines a $(B, A)$-bimodule homomorphism. It is clearly surjective. Conversely define a map $\varphi': M\otimes_C (A'/\pi(I))\rightarrow M\otimes_C A'$ by setting $\varphi'(m\otimes \bar(n))=m\otimes n$, where $n$ is such that $\psi(n)=\bar{n}$. It is well defined since if $n'\neq n$ are such that $\psi(n')=\psi(n)$, then $n'-n\in \pi(I)$, hence one has for any $a\in I$ with $\pi(a)=n'-n$:
$$m\otimes n'-m\otimes n=(m\otimes 1)\cdot a=0,$$
because $a$ lies in the annihilator of $M\otimes_C A'$. The map $\varphi$ is a morphism of $(B,A)$-bimodules which is an inverse to $\varphi$. The proof works in the graded case thanks to lemma \ref{lem:stroppel} and thanks to the fact that $A'/\pi(I)$ inherits a grading from $A'$ because $\ker\pi$ is homogeneous and the morphisms we defined are all homogeneous of degree $0$.
\end{proof} 
\begin{lem}\label{lem:tensorgraded}
Let $W\in\mathcal{V}_n$. Then $\mathcal{O}(W)$ is graded. 
\end{lem}
\begin{proof}
Since $W$ is a union of Weyl lines its vanishing ideal is homogeneous as it is an intersection of ideals of lines (which are known to be homogeneous).
\end{proof}
\begin{rmq}\label{rmq:grad}
Putting \ref{lem:stroppel}, \ref{lem:annulateur} and \ref{lem:tensorgraded} together we have the following: if $M\in \bar{R}-\mathrm{mod}_{\mathbb{Z}}-\bar{R}$, $W\in\mathcal{V}_n$ and if the right operation of $\bar{R}$ on $M$ factors through $\mathcal{O}(V_M)$ where $V_M\in\mathcal{V}_n$ (in other words, $M$ can be viewed in $\bar{R}-\mathrm{mod}-\mathcal{O}(V_M)$), then $M$ lies in $\bar{R}-\mathrm{mod}_{\mathbb{Z}}-\mathcal{O}(V_M)$ and
$$B:=M\otimes_{\mathcal{O}(V_M)} \mathcal{O}(V_M\cap W)$$
\noindent lies in $\bar{R}-\mathrm{mod}_{\mathbb{Z}}-\bar{R}$.
\end{rmq}
\begin{lem}\label{lem:bi}
Let $i\in\{1,\dots,n\}$. The bimodule $B_i:=R_i\otimes_{R_i^{s_i}} R_i$ is graded. It is free as left $R_i$-module and as right $R_i$-module. 
\end{lem}
\begin{proof}
Since $s_i$ preserves the degrees $R_i^{s_i}$ is a graded subring of $R_i$ and so $R_i$ lies in $R_i-\mathrm{mod}_{\mathbb{Z}}-R_i^{s_i}$ and in $R_i^{s_i}-\mathrm{mod}_{\mathbb{Z}}-R_i$. Then apply lemma \ref{lem:stroppel}.

The fact that the bimodule $B_i$ is free as left $R_i$-module and as right $R_i$-module is a consequence of the decomposition $R_i=R_i^{s_i}\oplus R_i^{s_i} f_i$
\end{proof}
The bimodules $B_i$ as defined in the above lemma are the equivalent of the Soergel bimodules $R\otimes_{R^{s}} R$ used in \cite{S} to categorify the Kazhdan-Lusztig basis of the Hecke algebra of an arbitrary Coxeter system of finite rank. 
\subsection{Elementary bimodules}

\begin{lem}\label{lem:free}
The ring $R_{i,i+1}$ of regular functions on $V_i\cap V_{i+1}$ is a free $R_i^{s_i}$-module of rank $1$ and a free $R_{i+1}^{s_{i+1}}$-module of rank $1$. 
\end{lem}
\begin{proof}
Since $V_i$ is $s_i$-stable we have a decomposition $R_i=R_i^{s_i}\oplus R_i^{s_i} f_i$. Since $R_i\twoheadrightarrow R_{i,i+1}$ if follows that $R_{i,i+1}$ is generated by $1$ and $f_i$ as a $R_i^{s_i}$-module. Thanks to the preceding lemma, $V_i\cap V_{i+1}\subset H_{s_i s_{i+1} s_i}$ and hence $f_i+f_{i+1}=0$ in $R_{i,i+1}$. It follows that the element $2f_{i+1}+f_i\in R_i^{s_i}$ applied on $1\in R_{i,i+1}$ yields $-f_i$ and hence that $R_{i,i+1}$ is generated as a $R_i^{s_i}$-module by $1$. It remains to show that if $f\in R_i^{s_i}$, $f\cdot 1=f|_{V_i\cap V_{i+1}}=0$ implies $f=0$. Since $f$ is $s_i$-invariant it is enough to show that
$$(V_i\cap V_{i+1})\cup s_i(V_i\cap V_{i+1})=V_i.$$
But this holds thanks to example \ref{ex:voisin}. The proof of the second statement is similar.
\end{proof}
\begin{cor}
As a left $R_i^{s_i}$-module, $R_{i,i+1}\otimes_{R_{i+1}^{s_{i+1}}} R_{i+1}$ is free of rank $2$. Similarly as a right $R_{i+1}^{s_{i+1}}$-module, $R_i\otimes_{R_i^{s_i}} R_{i,i+1}$ is free of rank $2$.
\end{cor}
\begin{proof}
Thanks to lemma \ref{lem:free}, $R_{i,i+1}\cong R_i^{s_i}$ as a left $R_i^{s_i}$-module. Since $R_ {i+1}=R_ {i+1}^{s_{i+1}}\oplus R_ {i+1}^{s_{i+1}} f_{i+1}$, the claim follows. 
\end{proof}
\begin{cor}\label{cor:nachbar}
The bimodule $B_{i,i+1}:=R_i\otimes_{R_i^{s_i}} R_{i,i+1} \otimes_{R_{i+1}^{s_{i+1}}} R_{i+1}$ which lies in $R_i-\mathrm{mod}_{\mathbb{Z}}-R_{i+1}$ is free of rank $2$ in $R_i-\mathrm{mod}$ and free of rank $2$ in $\mathrm{mod}-R_{i+1}$. In particular, if we view $B_{i,i+1}$ in $\bar{R}-\mathrm{mod}_{\mathbb{Z}}-\bar{R}$, then the left annihilator of $B_{i,i+1}$ in $\bar{R}$ is the ideal of functions vanishing on $V_i$ and its right annihilator is the ideal of functions vanishing on $V_{i+1}$.
\end{cor}
\begin{proof}
Thanks to the preceding lemma, $R_{i,i+1} \otimes_{R_{i+1}^{s_{i+1}}} R_{i+1}$ is free as a left $R_i^{s_i}$-module. Since $R_i=R_i^{s_i}\oplus R_i^{s_i} f_i$, it follows that $R_i\otimes_{R_i^{s_i}} R_{i,i+1} \otimes_{R_{i+1}^{s_{i+1}}} R_{i+1}$ is free as a left $R_i$-module. 
\end{proof}
We now study bimodules $B_{i,j}$ as defined in corollary \ref{cor:nachbar} but for $|i-j|>1$. Notice that $R_{i,j}\otimes_{R_j^{s_j}} R_j$ is free as left $R_{i,j}$-module since $R_j=R_j^{s_j}\oplus R_j^{s_j} f_j$. 
\begin{lem}
Any function $f\in R_j$ which vanishes on $V_i\cap V_j$ acts on $M:=R_ {i,j}\otimes_{R_j^{s_j}} R_j$ on the right by zero. In other words, the right operation of $R_j$ on $M$ gives rise to a right $R_{i,j}$-module structure on $M$. Moreover, $M$ is free as a right $R_{i,j}$-module.
\end{lem}
\begin{proof}
Decompose $f$ as $r+r'f_j$ with $r, r'\in R_j^{s_j}$. By assumption one has
$$r|_{V_i\cap V_j}+r'|_{V_i\cap V_j} f_j|_{V_i\cap V_j}=0.$$
Now since $|i-j|>1$, $V_i\cap V_j$ is $s_j$-stable, giving rise to a natural operation of $s_j$ on $R_{i,j}$. Applying $s_j$ to the above equation one gets
$$r|_{V_i\cap V_j}-r'|_{V_i\cap V_j} f_j|_{V_i\cap V_j}=0,$$
which implies that $r|_{V_i\cap V_j}=0$ and $r'|_{V_i\cap V_j} f_j|_{V_i\cap V_j}=0$. Since $f_j(v)\neq 0$ for $v\in V_i\cap V_j-\{0\}$, this forces $r'|_{V_i\cap V_j}=0$. Hence if $v\otimes w\in R_{i,j}\otimes_{R_j^{s_j}} R_j$, one gets
$$(v\otimes w)\cdot f=vr|_{V_i\cap V_j}\otimes w+vr'|_{V_i\cap V_j}\otimes wf_j=0.$$
To see that $M$ is free on the right over $R_{i,j}$, one first uses lemma \ref{lem:annulateur} to get an isomorphism $M\cong R_{i,j}\otimes_{R_j^{s_j}} R_{i,j}$ and then concludes by using the decomposition $R_{i,j}=R_{i,j}^{s_j}\oplus R_{i,j}^{s_j} f_j$ which holds since $V_i\cap V_j$ is $s_j$-invariant. 
\end{proof}
\begin{prop}\label{prop:distant}
The bimodule $B_{i,j}:=R_i\otimes_{R_i^{s_i}} R_{i,j}\otimes_{R_j^{s_j}} R_j$ which lies in $R_{i,j}-\mathrm{mod}_{\mathbb{Z}}-R_{i,j}$ thanks to the preceding lemma is free of rank $4$ as left $R_{i,j}$-module and as right $R_{i,j}$-module. It particular the left annihilator of $B_{i,j}$ is equal to its right annihilator and is the ideal of functions vanishing on $V_i\cap V_j$.
\end{prop}
\begin{proof}
As a left $R_i$-module, $B_{i,j}$ is generated by $t_1:=1\otimes 1\otimes 1$, $t_2:=1\otimes 1\otimes f_j$, $t_3:=1\otimes f_i\otimes f_j$ and $t_4:=1\otimes f_i\otimes 1$. Lets show that it is a basis of $B_{i,j}$ over $R_{i,j}$. Consider elements $a_k\in R_i$, $k=1, 2, 3, 4$ and write them as $a_k=r_k+r'_k f_i$ with $r_k, r'_k\in R_i^{s_i}$, $k=1,\dots,4$, and suppose $\sum_{i=1}^4 a_k\cdot t_k=0$. One gets
\begin{eqnarray*}
& &1\otimes (r_2+r_3 f_i)|_{V_i\cap V_j}\otimes f_j+1\otimes (r_1+r_4 f_i)|_{V_i\cap V_j}\otimes 1\\
& &+f_i\otimes (r_2'+r_3' f_i)|_{V_i\cap V_j}\otimes f_j+f_i\otimes(r_1'+r_4' f_i)|_{V_i\cap V_j}\otimes 1\\
& &=0.
\end{eqnarray*}

Now since $N:=R_i\otimes_{R_i^{s_i}} R_{i,j}$ is free as a right $R_{i,j}$-module and $M:=R_{i,j}\otimes_{R_j^{s_j}} R_j$ is free as a left $R_{i,j}$-module, $B_{i,j}=N\otimes_{R_{i,j}} M$ is free for the induced structure of $R_{i,j}$-module (which is not the same than its left or right $R_{i,j}$-module structure !), and a basis is given by $1\otimes 1\otimes 1$, $f_i\otimes 1\otimes 1$, $1\otimes 1\otimes f_j$ and $f_i\otimes 1\otimes f_j$. This implies that 
$$0=(r_1+r_3 f_i)|_{V_i\cap V_j}=(r_2+r_4 f_i)|_{V_i\cap V_j}=(r_1'+r_3' f_i)|_{V_i\cap V_j}=(r_2'+r_4' f_i)|_{V_i\cap V_j}.$$
Now the same argument as in the proof of the preceding lemma (applying $s_i$ this time) gives that $r_k|_{V_i\cap V_j}=0=r'_k|_{V_i\cap V_j}$, hence that $a_k|_{V_i\cap V_j}=0$ for all $k$, which concludes.
\end{proof}

\subsection{A product of bimodules}
Given two bimodules $B, B'\in\bar{R}-\mathrm{mod}_{\mathbb{Z}}-\bar{R}$, one defines a bimodule $B\ast B'$ in the following way : let $I_B^r$ be the right annihilator of $B$ and $I_{B'}^\ell$ the left annihilator of $B'$, and write $V_B^r$, $V_{B'}^{\ell}$ for the corresponding closed subvarieties of $Z$. Then set
$$B\ast B':=B\otimes_{\bar{R}} \mathcal{O}(V_B^r\cap V_{B'}^{\ell})\otimes_{\bar{R}} B'.$$
\noindent We will often omit the exponents $\ell$ and $r$ when no confusion is possible. Thanks to remark \ref{rmq:grad}, such a bimodule lies in $\bar{R}-\mathrm{mod}_{\mathbb{Z}}-\bar{R}$ in case all the varieties occuring in its definition are union of Weyl lines. Note that if $B, B'$ have trivial right, respectively left annihilators (for example if they are free as right, resp. left $\bar{R}$-modules), this product is nothing but a tensor product over $\bar{R}$.
\begin{rmq}
In all the cases we will consider further, we will always have $I_B^r=I(V_B^r)$ and $I_{B'}^\ell=I(V_{B'}^{\ell})$. We will therefore often write the $\ast$-product as 
$$B\otimes_{\mathcal{O}(V_B^r)} \mathcal{O}(V_B^r\cap V_{B'}^{\ell})\otimes_{\mathcal{O}(V_{B'}^{\ell})} B'.$$
\end{rmq}
\noindent Recall the bimodules $B_i:=R_i\otimes_{R_i^{s_i}} R_i$ with $i\in\{1,\dots,n\}$ from lemma \ref{lem:bi}.
\begin{lem}\label{lem:annulateurlibre}
Let $M$ be a right $\bar{R}$-module which is free over $\mathcal{O}(V_M)$ for $V_M\in\mathcal{V}_n$. The right annihilator of $M\ast B_i$ is the ideal of the variety 
$$V_i\cap (V_M\cup s_i V_M).$$
Moreover $M\ast B_i$ is free as a right $\mathcal{O}(V_i\cap (V_M\cup s_i V_M))$-module. The same statement holds for the left operation on a bimodule $B_i\ast M$ in case $M$ is a left $\bar{R}$-module which is free over $\mathcal{O}(V_M)$.
\end{lem} 
\begin{proof}
Let $f\in\bar{R}$ and annihilate $M\ast B_i$ on the right. One can suppose $f\in R_i$. Write $f=r+r' f_i$ with $r, r'\in R_i^{s_i}$. We can suppose $M\ast B_i\cong \mathcal{O}(V_M\cap V_i)\otimes_{R_i^{s_i}} R_i$ (as a right $\bar{R}$-module) since $M$ is free as a right $\mathcal{O}(V_M)$-module.

Now if $f=r+r'f_i$, $r, r'\in R_i^{s_i}$ annihilates $M\ast B_i$, in particular it annihilates $1\otimes 1$. Hence one has
$$r|_{V_M\cap V_i}\otimes 1+r'|_{V_M\cap V_i}\otimes f_i=0.$$
This forces $r|_{V_M\cap V_i}=0=r'|_{V_M\cap V_i}$ (because $\mathcal{O}(V_M\cap V_i)\otimes_{R_i^{s_i}} R_i$ is free as a module over $\mathcal{O}(V_M\cap V_i)$ for the obvious operation). Since $r, r'$ are $s_i$-invariant, this forces them to be zero on $V_i\cap (V_M\cup s_i V_M)$, and the same holds for $f$. Conversely if $f\in R_i$ is zero on $V_i\cap(V_M\cup s_i V_M)$, then write $f=r+r'f_i$ with $r, r'$ invariant under $s_i$. This forces $r, r'$ to be zero on $V_i\cap(V_M\cup s_i V_M)$.

Now show it is free ; first suppose $V_M\cap V_i$ is $s_i$-invariant ; hence $\mathcal{O}(V_M\cap V_i)= \mathcal{O}(V_M\cap V_i)^{s_i}\oplus\mathcal{O}(V_M\cap V_i)^{s_i} f_i$. It follows that $\mathcal{O}(V_M\cap V_i)\otimes_{R_i^{s_i}} R_i$ is generated as a right $\bar{R}$-module by $1\otimes 1$ and $f_i\otimes 1$. Let $r, r'\in R_i$ be such that
$$1\otimes r+f_i\otimes r'=0.$$
Write $r=r_1+r_2 f_i$ and $r'=r'_1+r'_2 f_i$ with $r_j, r'_j\in R_i^{s_i}$ and get
$$(r_1+r'_1 f_i)|_{V_M\cap V_i}\otimes 1+(r_2+r'_2 f_i)|_{V_M\cap V_i}\otimes f_i=0.$$
This implies that $(r_1+r'_1 f_i)|_{V_M\cap V_i}=0=(r_2+r'_2 f_i)|_{V_M\cap V_i}$ and by invariance one gets $r'_j|_{V_M\cap V_i}=0=r_j|_{V_M\cap V_i}$ for $j=1, 2$. Hence $M\ast B_i$ is free on the right over $\mathcal{O}(V_M\cap V_i)$, of rank $2$.

Now suppose $V_i\cap V_M$ is not $s_i$-invariant. Remark \ref{rmq:noninvariant} implies that $R_i^{s_i}\twoheadrightarrow\mathcal{O}(V_M\cap V_i)$. Hence as a right $\mathcal{O}(V_i\cap(V_M\cup s_i V_M))$-module, $\mathcal{O}(V_M\cap V_i)\otimes_{R_i^{s_i}} R_i$ is generated by $1\otimes 1$. We have to show that if $f\in R_i$, $1\otimes f=0$ implies that $f|_{V_i\cap(V_M\cup s_i V_M)}=0$. Write $f=r+r'f_i$ with $r, r'\in R_i^{s_i}$. This implies that $r'|_{V_M\cap V_i}=0=r|_{V_M\cap V_i}$. Now since $r', r$ are $s_i$-invariant one concludes that they also vanish on $V_i\cap(V_M\cup s_i V_M)$ and the same holds for $f$. 
\end{proof}
In particular, if a module $M$ has as right annihilator $I(W)$ with $W\in\mathcal{V}_n$, then $M\ast B_i$ has as right annihilator $I(s_i\cdot W)$ and by definition $s_i\cdot W\in\mathcal{V}_n$. The above lemma will allow us to use induction.  

\subsection{Associativity}
Unfortunately, the product defined in the previous section is not associative for arbitrary bimodules $B, B'$. However, as we will see in this section, it will be associative when restricted to a suitable family of bimodules, exactly the bimodules occuring by considering successive $\ast$-products of the bimodules $B_i$, $i\in\{1,\dots,n\}$.
 
A first step in proving the associativity of the product $\ast$ is to prove the following : If $M, N\in\bar{R}-\mathrm{mod}_{\mathbb{Z}}-\bar{R}$ with $M$ having $I(V_M)$ as right annihilator and $N$ having $I(V_N)$ as left annihilator, then 
\begin{eqnarray}\label{eq:associativity}
(M\ast B_i)\ast N\cong M\ast (B_i\ast N)
\end{eqnarray}
provided $V_N, V_M$ lie in a certain family of subvarieties of $Z$ ; thanks to lemma \ref{lem:annulateurlibre} the good family to choose is $\mathcal{V}_n$. The idea will be then to show associativity of the $\ast$ product for products of three of the bimodules $B_i$ and then use this previous result to generalise to arbitrary products of the $B_i$. 

Let's rewrite equation \ref{eq:associativity}. We suppose that $M$ is free on the right over $\mathcal{O}(V_M)$ and that $N$ is free on the left over $\mathcal{O}(V_N)$. Set $W_{i,M}:=V_i\cap(V_M\cup s_i V_M)$, $W_{i,N}:=V_i\cap(V_N\cup s_i V_N)$. By definition of the $\ast$ product together with lemma \ref{lem:annulateurlibre} the left hand side of \ref{eq:associativity} can be rewritten as
$$(M\otimes_{\mathcal{O}(V_M)} \mathcal{O}(V_M\cap V_i)\otimes_{R_i}(R_i\otimes_{R_i^{s_i}} R_i))\otimes_{\mathcal{O}(W_{i,M})} \mathcal{O}(W_{i,M}\cap V_N)\otimes_{\mathcal{O}(V_N)} N,$$ or shorter 
$$(M\otimes_{\mathcal{O}(V_M)} \mathcal{O}(V_M\cap V_i)\otimes_{R_i^{s_i}} R_i)\otimes_{\mathcal{O}(W_{i,M})} \mathcal{O}(W_{i,M}\cap V_N)\otimes_{\mathcal{O}(V_N)} N.$$ Now using lemmas \ref{lem:annulateurlibre} and \ref{lem:annulateur} we can rewrite this as
$$(M\otimes_{\mathcal{O}(V_M)} \mathcal{O}(V_M\cap V_i)\otimes_{R_i^{s_i}} \mathcal{O}(W_{i,M}))\otimes_{\mathcal{O}(W_{i,M})} \mathcal{O}(W_{i,M}\cap V_N)\otimes_{\mathcal{O}(V_N)} N.$$ or shorter
$$M\otimes_{\mathcal{O}(V_M)} \mathcal{O}(V_M\cap V_i)\otimes_{R_i^{s_i}} \mathcal{O}(W_{i,M}\cap V_N)\otimes_{\mathcal{O}(V_N)} N.$$ Doing the same reductions for the right hand side one gets
$$M\otimes_{\mathcal{O}(V_M)} \mathcal{O}(V_M\cap W_{i,N})\otimes_{R_i^{s_i}} \mathcal{O}(V_i\cap V_N)\otimes_{\mathcal{O}(V_N)} N.$$
\noindent Now our job is to show that these two bimodules are isomorphic in $\bar{R}-\mathrm{mod}_{\mathbb{Z}}-\bar{R}$. It is therefore enough to show that 
$$\mathcal{O}(V_M\cap V_i)\otimes_{R_i^{s_i}} \mathcal{O}(W_{i,M}\cap V_N)\cong\mathcal{O}(V_M\cap W_{i,N})\otimes_{R_i^{s_i}} \mathcal{O}(V_i\cap V_N),$$
where the isomorphism holds in $\mathcal{O}(V_M)-\mathrm{mod}_{\mathbb{Z}}-\mathcal{O}(V_N)$.
\begin{prop}\label{prop:assob?b?}
One has
$$\mathcal{O}(V_M\cap V_i)\otimes_{R_i^{s_i}} \mathcal{O}(W_{i,M}\cap V_N)\cong\mathcal{O}(V_M\cap W_{i,N})\otimes_{R_i^{s_i}} \mathcal{O}(V_i\cap V_N),$$
as graded $(\mathcal{O}(V_M),\mathcal{O}(V_N))$-bimodules. 
\end{prop}
\begin{proof}
The strategy is to find the left and right annihilators and then use lemma \ref{lem:annulateur}. We first suppose $V_N\cap V_i$ is $s_i$-invariant. Hence $W_{i,M}\cap V_N$ is $s_i$-invariant. Let $g\in\mathcal{O}(V_N\cap V_i)$ be such that $g|_{V_N\cap W_{i,M}}=0$. Chose $h\in R_i$, $h=r+r'f_i$ with $r, r'\in R_i^{s_i}$ such that $h|_{V_N\cap V_i}=g$. Since $V_N\cap W_{i,M}$ is $s_i$-invariant one has that $r'|_{V_N\cap W_{i,M}}=0=r|_{V_N\cap W_{i,M}}$, hence also $r'|_{V_M\cap W_{i,M}}=0=r|_{V_M\cap W_{i,M}}$ since in our case $V_M\cap W_{i,N}\hookrightarrow V_N\cap W_{i,M}$ (because of $s_i$-invariance of $V_N\cap V_i$). We have shown that an element $g\in\mathcal{O}(V_N\cap V_i)$ which vanishes on $V_N\cap W_{i,M}$ kills $\mathcal{O}(V_M\cap W_{i,N})\otimes_{R_i^{s_i}} \mathcal{O}(V_i\cap V_N)$ on the right, hence by lemma \ref{lem:annulateur}, the right hand side is isomorphic to 
$$\mathcal{O}(V_M\cap W_{i,N})\otimes_{R_i^{s_i}} \mathcal{O}(V_N\cap W_{i,M}).$$
Now if $V_M\cap V_i$ is $s_i$-invariant one uses the same argument for the left hand side for the left operation and this left hand side is isomorphic to
$$\mathcal{O}(V_M\cap W_{i,N})\otimes_{R_i^{s_i}} \mathcal{O}(V_N\cap W_{i,M}),$$
\noindent which concludes. 

Now suppose $V_M\cap V_i$ is not $s_i$-invariant. Consider $g\in \mathcal{O}(V_M\cap V_i)$ vanishing on $X:=V_M\cap W_{i,N}$. By assumption $V_M$ lies in $\mathcal{V}_n$ and thanks to remark \ref{rmq:noninvariant}, one can choose $h\in R_i^{s_i}$ such that $h|_{V_M\cap V_i}=g$. In particular $h|_X=0$. Now since $h$ is $s_i$-invariant it has to vanish on $X\cup s_i X$. But $V_N\cap W_{i,M}\hookrightarrow X\cup s_i X$. Hence $h$, whence $g$ kills $\mathcal{O}(V_M\cap V_i)\otimes_{R_i^{s_i}} \mathcal{O}(W_{i,M}\cap V_N)$ on the left, and this bimodule is hence isomorphic to 
$$\mathcal{O}(V_M\cap W_{i,N})\otimes_{R_i^{s_i}} \mathcal{O}(V_N\cap W_{i,M})$$
\noindent thanks to lemma \ref{lem:annulateur}. The case where $V_N\cap V_i$ is not $s_i$-invariant but $V_M\cap V_i$ is is symmetric ; in case none of them is $s_i$-invariant, the argument given above (choose a preimage $h$ which is invariant and then restrict) can still be given, for the left as well as for the right operation, since it makes no use of the fact that the variety on the other side is $s_i$-invariant or not.  
\end{proof}
We define bimodules associated to finite sequences of integers in $[1,n]$. If the sequence has length $1$, containing a single index $j$, the corresponding bimodule is $B_j$. Let $i_1,\dots, i_k\in[1,n]$. Define the bimodule associated to this sequence by setting $B_{i_k\cdots i_1}=B_{i_k}\ast B_{i_{k-1}\cdots i_1}$. A bimodule $B$ will be said to be \textbf{associated} to such a sequence if it is obtained from $B_{i_k},\dots, B_{i_1}$ by doing a product in this order but with a possibly different choice of brackets from the one we made for $B_{i_k\cdots i_1}$. For example, $(B_{i_4}\ast B_{i_3})\ast (B_{i_2}\ast B_{i_1})$ and $B_{i_4}\ast ((B_{i_3}\ast B_{i_2})\ast B_{i_1})$ are associated to the same sequence  $i_4\cdots i_1$. 
\begin{thm}\label{thm:asso}
Let $i_k\cdots i_1$ be a sequence of indices in $\{1,\dots,n\}$. 
\begin{enumerate}
	\item Two bimodules associated to this sequence are isomorphic in $\bar{R}-\mathrm{mod}_{\mathbb{Z}}-\bar{R}$.
	\item The bimodule $B_{i_k\cdots i_1}$ is free on the left on $\mathcal{O}(W_{i_k\cdots i_1})$ and free on the right on $\mathcal{O}(W_{i_1\cdots i_k})$
\end{enumerate}

\end{thm}
\begin{proof}
Both properties are proved simultaneously by using induction on the number of elementary bimodules $B_i$ occuring in a product. If our bimodule is a product of three of the $B_i$, say $(B_i\ast B_j)\ast B_k$, then associativity is immediate by proposition \ref{prop:assob?b?} and the arguments above it : one has
$$(B_i\ast B_j)\ast B_k\cong B_i\ast(B_j\ast B_k),$$
and both of theses bimodules are free as left $\mathcal{O}(W_{ijk})$-modules and as right $\mathcal{O}(W_{kji})$-modules thanks to corollary \ref{cor:nachbar}, proposition \ref{prop:distant} and lemma \ref{lem:annulateurlibre}. 
Now suppose the result holds for any product of at most $m-1$ of the $B_i$'s. Consider a sequence $i_1,\dots,i_{m}\in\{1,\dots,n\}$. By induction it is enough to show that 
$$(B_{i_1}\ast\cdots\ast B_{i_j})\ast(B_{i_{j+1}}\ast\cdots B_{i_m})\cong (B_{i_1}\ast\cdots\ast B_{i_k})\ast(B_{i_{k+1}}\ast\cdots B_{i_m}),$$
with $k\neq j$, where by induction the products $B_{i_1}\ast\cdots\ast B_{i_j}$, $B_{i_{j+1}}\ast\cdots B_{i_m}$, $B_{i_1}\ast\cdots\ast B_{i_k}$ and $B_{i_{k+1}}\ast\cdots B_{i_m}$ are well defined up to isomorphism (they can be written without brackets) and free over the varieties associated to their sequences (on the left over $\mathcal{O}(W_{i_1\cdots i_j})$ and on the right over $\mathcal{O}(W_{i_j\cdots i_1})$ for the first one, ...). One just has to apply successively proposition \ref{prop:assob?b?} to move $B_j$'s from one bracket to the other one. In particular both our bimodules are isomorphic to
$$B_{i_1}\ast(B_{i_2}\ast\cdots\ast B_{i_k})~\text{and}~(B_{i_1}\ast\cdots\ast B_{i_{k-1}})\ast B_{i_k},$$
which are free by induction together with lemma \ref{lem:annulateurlibre}. In particular this lemma tells us that the left annihilator is $I(W_{i_1\cdots i_k})$ and the right one is $I(W_{i_k\cdots i_1})$. 

\end{proof}
\section{Realization of the Temperley-Lieb algebra}

\subsection{The Temperley-Lieb algebra}
Let $\tau$ be a formal parameter. The Temperley-Lieb algebra $\mathrm{TL}_n$ is the $\mathbb{Z}[\tau, \tau^{-1}]$-algebra generated by elements $b_{s_i}=b_i$ for $i=1,\dots,n$ with relations 
\begin{eqnarray*}
b_j b_i b_j=b_j~\text{if $|i-j|=1$},\\
b_i b_j=b_j b_i~\text{if $|i-j|>1$},\\
b_i^2=(1+\tau^{-2}) b_i.
\end{eqnarray*}
\begin{rmq}\label{rmq:parametre}
Usually $\mathrm{TL}_n$ is defined with a formal parameter $v$ instead of $\tau$, the last relation being replaced by $b_i^2=(v+v^{-1}) b_i$, which allows $\mathrm{TL}_n$ to be realized as a quotient of the Hecke algebra of type $A_n$. The reason for choosing another parameter $\tau$ is that the bimodules $B_i$ defined before will satisfy the above relations where the multiplication in $\mathrm{TL}_n$ corresponds to the $\ast$ product, the sum to direct sums of bimodules and the parameter $\tau$ to a shift. In the case of Soergel bimodules categorifying the Hecke algebra, one defines the analog of our bimodule $B_i$ by $S_i':=R\otimes_{R^{s_i}} R$ ; it turns out that the relation $S_i'\otimes_R S_i'\cong S_i'\oplus S_i'[-2]$ is satisfied but one then sets $S_i:=S_i'[1]$ and the relation becomes $S_i\otimes_R S_i\cong S_i[1]\oplus S_i[-1]$. The parameter $v$ is then interpreted as a shift and such a relation corresponds to the relation $C_{s_i}'^2=(v+v^{-1}) C_{s_i}'$ which holds in the Hecke algebra, $C_{s_i}'$ being the element of the Kazhdan-Lusztig basis (defined in \cite{KL}) indexed by the simple reflection $s_i$. In our case shifting the bimodules $B_i$ as in Soergel's work is a priori not possible since the first relation defining $\mathrm{TL}_n$ is not homogeneous.  
\end{rmq}
\begin{defn}
Let $(\mathcal{W}, S)$ be an arbitrary Coxeter system with $S$ finite. An element $w\in \mathcal{W}$ is \textbf{fully commutative} or \textbf{braid avoiding} if one can pass from any reduced expression for $w$ to any other only by applying relations of the form $st=ts$ for $s, t\in S$.
\end{defn}
The set of fully commutative elements is denoted by $\mathcal{W}_c$. Now if $(\mathcal{W}, S)$ is of type $A$ and $w\in \mathcal{W}_c$ and $t_1\cdots t_k$ is a reduced expression for $w$, one can show that the element $b_w:=b_{t_1}\cdots b_{t_k}\in \mathrm{TL}_n$ is independent of the choice of the reduced expression for $w$ and that the set $\{b_w\}_{w\in\mathcal{W}_c}$ spans $\mathrm{TL}_n$ as a $\mathbb{Z}[\tau,\tau^{-1}]$-module.

\begin{defn}
The basis $\{b_w\}_{w\in\mathcal{W}_c}$ of $\mathrm{TL}_n$ is the \textbf{Kazhdan-Lusztig basis} of $\mathrm{TL}_n$.
\end{defn}
\begin{rmq}
This vocabulary is due to the fact that if we define $\mathrm{TL}_n$ algebra with a parameter $v$ instead of $\tau$ as mentioned in remark \ref{rmq:parametre} it is a quotient of $H_n$, the Hecke algebra of type $A_n$, and if $w\in\mathcal{W}_c$, the image in the quotient of the element $C_w'$ of the Kazhdan-Lusztig basis of $H_n$ is $b_w$ and any element $C_x'$ for $x\notin\mathcal{W}_c$ is sent to zero (see \cite{FG}, Theorem 3.8.2 for type $A$ or \cite{GL} for other types). 
\end{rmq}
The basis $\{b_w\}_{w\in\mathcal{W}_c}$ has a well-known interpretation by planar diagrams. Draw a sequence of $n+1$ points on a line and another one under the first one. Draw arcs between any two points of the two sequences (the two points of an arc can be on the same sequence) such that each point occurs in exactly one arc and such that two distinct arcs never cross to obtain a diagram like the one given in figure \ref{figure:diagramme} ; we always consider such diagrams up to isotopy. Elements of the Temperley-Lieb algebra are $\mathbb{Z}[\tau, \tau^{-1}]$-linear combinations of such diagrams, where the element $b_i=b_{s_i}$ is given by the diagram in figure \ref{figure:tli}.\begin{figure}
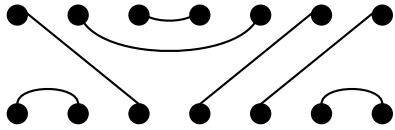

\begin{center}
\begin{psmatrix}[colsep=0.8,rowsep=0.8]
\pscircle*{0.14} & \pscircle*{0.14} & \pscircle*{0.14} & \pscircle*{0.14} & 
\pscircle*{0.14} & \pscircle*{0.14} & \pscircle*{0.14}\\
\pscircle*{0.14} & \pscircle*{0.14} & \pscircle*{0.14} & \pscircle*{0.14} & 
\pscircle*{0.14} & \pscircle*{0.14} & \pscircle*{0.14}\\
\end{psmatrix}
\ncline{1,1}{2,3}
\ncarc[arcangle=90]{2,1}{2,2}
\ncarc[arcangle=-90]{1,2}{1,5}
\ncarc[arcangle=-90]{1,3}{1,4}
\ncline{1,6}{2,4}
\ncline{1,7}{2,5}
\ncarc[arcangle=90]{2,6}{2,7}
\end{center}
\caption{A diagram representing an element of the Temperley-Lieb algebra} 
\label{figure:diagramme}
\end{figure}
\begin{figure}
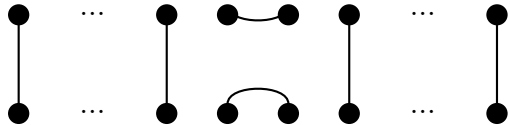

\begin{center}
\begin{psmatrix}[colsep=0.8,rowsep=0.8]
\pscircle*{0.14} & ... & \pscircle*{0.14} & \pscircle*{0.14} & 
\pscircle*{0.14} & \pscircle*{0.14} & ... & \pscircle*{0.14}\\
\pscircle*{0.14} & ... & \pscircle*{0.14} & \pscircle*{0.14} & 
\pscircle*{0.14} & \pscircle*{0.14} & ... & \pscircle*{0.14}\\
\end{psmatrix}
\ncline{1,1}{2,1}
\ncarc[arcangle=-90]{1,4}{1,5}
\ncarc[arcangle=90]{2,4}{2,5}
\ncline{1,3}{2,3}
\ncline{1,6}{2,6}
\ncline{1,8}{2,8}
\end{center}
\caption{Planar diagram corresponding to the element $b_i$, where the two arcs go from index $i$ to index $i+1$.} 
\label{figure:tli}
\end{figure}
Multiplication of two planar diagrams is then given by concatenating the diagrams ; if circles occur in the resulting diagram, we remove them and multiply the diagram by $(1+\tau^{-2})^k$ where $k$ is the number of circles. The diagram algebra over $\mathbb{Z}[\tau, \tau^{-1}]$ obtained in this way turns out to be isomorphic to $\mathrm{TL}_n$.  

\subsection{Temperley-Lieb relations}
The aim of this section is to prove that the bimodules $B_i$ together with the $\ast$ product from the previous section satisfy the Temperley-Lieb relations, i.e.,
\begin{eqnarray*}
B_j\ast B_i\ast B_j\cong B_j~\text{if $|i-j|=1$},\\
B_i\ast B_j\cong B_j\ast B_i~\text{if $|i-j|>1$},\\
B_i\ast B_i\cong B_i\oplus B_i[-2],
\end{eqnarray*}
\noindent where all the isomorphisms hold in $\bar{R}-\mathrm{mod}_{\mathbb{Z}}-\bar{R}$. 
\begin{thm}\label{thm:relations}
The bimodules $B_i$ satisfy the Temperley-Lieb relations.
\end{thm}
\begin{proof}
For short we write $R_i:=\mathcal{O}(V_i)$, $R_{i,j}=\mathcal{O}(V_i\cap V_j)$. For the first relation, suppose $j=i+1$, the other case being similar. The left hand side of the first relation which is isomorphic to $(B_j\ast B_i)\ast B_j$ can be rewritten thanks to corollary \ref{cor:nachbar}
$$(R_{i+1}\otimes_{R_{i+1}^{s_{i+1}}} R_{i+1}\otimes_{R_{i+1}} R_{i,i+1}\otimes_{R_i} R_i\otimes_{R_i^{s_i}} R_i)\otimes_{R_i} R_{i,i+1}\otimes_{R_{i+1}} (R_{i+1} \otimes_{R_{i+1}^{s_{i+1}}} R_{i+1}),$$ which is isomorphic to 
$$R_{i+1}\otimes_{R_{i+1}^{s_{i+1}}} R_{i,i+1}\otimes_{R_i^{s_i}} R_{i,i+1} \otimes_{R_{i+1}^{s_{i+1}}} R_{i+1}.$$
Hence it suffices to show that $R_{i,i+1}\otimes_{R_i^{s_i}} R_{i,i+1}\cong R_{i+1}^{s_{i+1}}$ as graded $(R_{i+1}^{s_{i+1}}, R_{i+1}^{s_{i+1}})$-bimodule. But $R_{i+1}^{s_{i+1}}$ is known to be isomorphic to $R_{i,i+1}$ thanks to lemma \ref{lem:free} (the left and right operations are the same hence this is a bimodule isomorphism). Define a map
$$\varphi: R_{i,i+1}\otimes_{R_i^{s_i}} R_{i,i+1}\rightarrow R_{i,i+1}$$
$$a\otimes b\mapsto ab.$$
This clearly defines a morphism of bimodules. Define a map 
$$\psi: R_{i,i+1}\rightarrow R_{i,i+1}\otimes_{R_i^{s_i}} R_{i,i+1} $$
$$c\mapsto c\otimes 1.$$
One checks using lemma \ref{lem:free} that this defines a morphism of bimodules which is an inverse to $\varphi$. Hence the first Temperley-Lieb relation holds.

For the second relation, using proposition \ref{prop:distant} and \ref{lem:annulateur}, it is enough to show that 
$$R_{i,j}\otimes_{R_i^{s_i}} R_{i,j}\otimes_{R_j^{s_j}} R_{i,j}\cong R_{i,j}\otimes_{R_j^{s_j}} R_{i,j}\otimes_{R_i^{s_i}} {R_{i,j}}$$
as graded $(R_{i,j}, R_{i,j})$-bimodules. Let $m, n, q\in R_{i,j}$. Since $V_i\cap V_j$ is $s_i$-invariant one has that $R_{i,j}=R_{i,j}^{s_i}\otimes R_{i,j}^{s_i} f_i$ ; write $n=r+r' f_i$ with $r, r'\in R_{i,j}^{s_i}$. Define a map 
$$\varphi : R_{i,j}\otimes_{R_i^{s_i}} R_{i,j}\otimes_{R_j^{s_j}} R_{i,j}\rightarrow R_{i,j}\otimes_{R_j^{s_j}} R_{i,j}\otimes_{R_i^{s_i}} \otimes R_{i,j}$$
$$m\otimes n\otimes q\mapsto mr\otimes 1\otimes q+mr'\otimes 1\otimes f_i q.$$
It is routine to check that such a map is well-defined and that it is a morphism of graded bimodules. By permuting the indices $i$ and $j$ one also gets a map $\psi$ in the other direction and one shows that $\psi$ is an inverse of $\varphi$. \\
\\
For the third relation one has to show that
$$R_i\otimes_{R_i^{s_i}} R_i\otimes_{R_i^{s_i}} R_i\cong (R_i\otimes_{R_i^{s_i}} R_i)\oplus (R_i\otimes_{R_i^{s_i}} R_i)[-2].$$
Now $R_i=R_i^{s_i}\oplus R_i^{s_i} f_i$ and since no irreducible component of $V_i$ is included in $H_{s_i}$ one has an $R_i^{s_i}$-(bi)module isomorphism $R_i^{s_i} f_i\cong R_i^{s_i}[-2]$ given by the restriction of the Demazure operator $\partial_{s_i}$ (which has in this case multiplication by $f_i$ as inverse). Hence $R_{i}\cong R_i^{s_i}\oplus R_i^{s_i}[-2]$ as graded $(R_i^{s_i}, R_i^{s_i})$-bimodule and one gets the claim by decomposing in such a way the $R_i$ in the middle of the above tensor product on the left hand side.
\end{proof}
\begin{defn}
Let $w\in\mathcal{W}_c$. Let $s_{i_1}\cdots s_{i_k}$ be a reduced expression for $w$. We consider the bimodule 
$$B_{i_1}\ast\cdots\ast B_{i_k}\in\bar{R}-\mathrm{mod}_{\mathbb{Z}}-\bar{R}.$$ Since bimodules $B_i$ satisfy the Temperley-Lieb relations, this bimodule is independent up to isomorphism of the choice of a reduced expression for $w$ and we label by $B_w$ any bimodule isomorphic to it in $\bar{R}-\mathrm{mod}_{\mathbb{Z}}-\bar{R}$. Such a bimodule $B_w$ will be called \textbf{fully commutative}. 
\end{defn}
\subsection{Link with dense sets of reflections}
For each fully commutative element $w\in\mathcal{W}_c$, one can consider the dense sets $T(i_1\cdots i_k)$ and $T(i_k\cdots i_1)$ where $s_{i_1}\cdots s_{i_k}$ is a reduced expression for $w$ ; such sets characterize the varieties whose ideals are the left and right annihilators in $\bar{R}$ of the bimodule $B_{i_1}\ast\cdots\ast B_{i_k}$. We have another way of associating a pair of dense sets to $w$:
\begin{nota} Let $w\in\mathcal{W}_c$ and consider the planar diagram corresponding to the element $b_w\in\mathrm{TL}_n$; if we remove the lines joining a point in the sequence at the top of the diagram to a point in the sequence at the bottom, we obtain a dense set at the top of the diagram, which we write $Q(i_1\cdots i_k)$, and a dense set at the bottom which we can write $Q(i_k\cdots i_1)$ since it is equal to the dense set obtained at the top of the diagram of $b_{w^{-1}}$ after applying the same process of removing lines going from the top to the bottom of the diagram (notice that $w^{-1}$ lies in $\mathcal{W}_c$ if and only if $w$ does).
\end{nota}
\begin{figure}
\begin{center}
\begin{tabular}{lr}

\begin{tabular}{|c|c|c|}
\hline
\begin{psmatrix}[colsep=0.35,rowsep=1.2]
\pscircle*{0.08} & \pscircle*{0.08} & \pscircle*{0.08} & \pscircle*{0.08} & 
\pscircle*{0.08} 
\end{psmatrix} & $e$ & \begin{psmatrix}[colsep=0.35,rowsep=1.2]
\pscircle*{0.08} & \pscircle*{0.08} & \pscircle*{0.08} & \pscircle*{0.08} & 
\pscircle*{0.08} 
\end{psmatrix}\\
& & \\
\hline
\begin{psmatrix}[colsep=0.35,rowsep=1.2]
\pscircle*{0.08} & \pscircle*{0.08} & \pscircle*{0.08} & \pscircle*{0.08} & 
\pscircle*{0.08}\\ 
\end{psmatrix}
\ncarc[arcangle=90]{1,1}{1,2} & $s_1$ & \begin{psmatrix}[colsep=0.35,rowsep=1.2]
\pscircle*{0.08} & \pscircle*{0.08} & \pscircle*{0.08} & \pscircle*{0.08} & 
\pscircle*{0.08}\\ 
\end{psmatrix}
\ncarc[arcangle=90]{1,1}{1,2}\\
& & \\
\hline
\begin{psmatrix}[colsep=0.35,rowsep=3]
\pscircle*{0.08} & \pscircle*{0.08} & \pscircle*{0.08} & \pscircle*{0.08} & 
\pscircle*{0.08}\\ 
\end{psmatrix}
\ncarc[arcangle=90]{1,2}{1,3} & $s_2$ & \begin{psmatrix}[colsep=0.35,rowsep=1.2]
\pscircle*{0.08} & \pscircle*{0.08} & \pscircle*{0.08} & \pscircle*{0.08} & 
\pscircle*{0.08}\\ 
\end{psmatrix}
\ncarc[arcangle=90]{1,2}{1,3}\\
& & \\
\hline
\begin{psmatrix}[colsep=0.35=0.5,rowsep=1.2]
\pscircle*{0.08} & \pscircle*{0.08} & \pscircle*{0.08} & \pscircle*{0.08} & 
\pscircle*{0.08}\\ 
\end{psmatrix}
\ncarc[arcangle=90]{1,3}{1,4} & $s_3$ & \begin{psmatrix}[colsep=0.35,rowsep=1.2]
\pscircle*{0.08} & \pscircle*{0.08} & \pscircle*{0.08} & \pscircle*{0.08} & 
\pscircle*{0.08}\\ 
\end{psmatrix}
\ncarc[arcangle=90]{1,3}{1,4}\\
& & \\
\hline
\begin{psmatrix}[colsep=0.35,rowsep=1.2]
\pscircle*{0.08} & \pscircle*{0.08} & \pscircle*{0.08} & \pscircle*{0.08} & 
\pscircle*{0.08}\\ 
\end{psmatrix}
\ncarc[arcangle=90]{1,4}{1,5} & $s_4$ & \begin{psmatrix}[colsep=0.35,rowsep=1.2]
\pscircle*{0.08} & \pscircle*{0.08} & \pscircle*{0.08} & \pscircle*{0.08} & 
\pscircle*{0.08}\\ 
\end{psmatrix}
\ncarc[arcangle=90]{1,4}{1,5}\\
& & \\
\hline
\begin{psmatrix}[colsep=0.35,rowsep=1.2]
\pscircle*{0.08} & \pscircle*{0.08} & \pscircle*{0.08} & \pscircle*{0.08} & 
\pscircle*{0.08} 
\end{psmatrix}
\ncarc[arcangle=90]{1,1}{1,2} & $s_1 s_2$ & \begin{psmatrix}[colsep=0.35,rowsep=1.2]
\pscircle*{0.08} & \pscircle*{0.08} & \pscircle*{0.08} & \pscircle*{0.08} & 
\pscircle*{0.08} 
\end{psmatrix}
\ncarc[arcangle=90]{1,2}{1,3}\\
& & \\
\hline
\begin{psmatrix}[colsep=0.35,rowsep=1.2]
\pscircle*{0.08} & \pscircle*{0.08} & \pscircle*{0.08} & \pscircle*{0.08} & 
\pscircle*{0.08} 
\end{psmatrix}
\ncarc[arcangle=90]{1,2}{1,3} & $s_2 s_1$ & \begin{psmatrix}[colsep=0.35,rowsep=1.2]
\pscircle*{0.08} & \pscircle*{0.08} & \pscircle*{0.08} & \pscircle*{0.08} & 
\pscircle*{0.08}\\ 
\end{psmatrix}
\ncarc[arcangle=90]{1,1}{1,2}\\
& & \\
\hline
\begin{psmatrix}[colsep=0.35,rowsep=1.2]
\pscircle*{0.08} & \pscircle*{0.08} & \pscircle*{0.08} & \pscircle*{0.08} & 
\pscircle*{0.08}\\ 
\end{psmatrix}
\ncarc[arcangle=90]{1,1}{1,2}
\ncarc[arcangle=90]{1,3}{1,4} & $s_1 s_3$ & \begin{psmatrix}[colsep=0.35,rowsep=1.2]
\pscircle*{0.08} & \pscircle*{0.08} & \pscircle*{0.08} & \pscircle*{0.08} & 
\pscircle*{0.08}\\ 
\end{psmatrix}
\ncarc[arcangle=90]{1,1}{1,2}
\ncarc[arcangle=90]{1,3}{1,4}\\
& & \\
\hline
\begin{psmatrix}[colsep=0.35,rowsep=1.2]
\pscircle*{0.08} & \pscircle*{0.08} & \pscircle*{0.08} & \pscircle*{0.08} & 
\pscircle*{0.08}\\ 
\end{psmatrix}
\ncarc[arcangle=90]{1,1}{1,2}
\ncarc[arcangle=90]{1,4}{1,5} & $s_1 s_4$ & \begin{psmatrix}[colsep=0.35,rowsep=1.2]
\pscircle*{0.08} & \pscircle*{0.08} & \pscircle*{0.08} & \pscircle*{0.08} & 
\pscircle*{0.08}\\ 
\end{psmatrix}
\ncarc[arcangle=90]{1,1}{1,2}
\ncarc[arcangle=90]{1,4}{1,5}\\
& & \\
\hline
\begin{psmatrix}[colsep=0.35,rowsep=1.2]
\pscircle*{0.08} & \pscircle*{0.08} & \pscircle*{0.08} & \pscircle*{0.08} & 
\pscircle*{0.08}\\ 
\end{psmatrix}
\ncarc[arcangle=90]{1,2}{1,3}
\ncarc[arcangle=90]{1,4}{1,5} & $s_2 s_4$ & \begin{psmatrix}[colsep=0.35,rowsep=1.2]
\pscircle*{0.08} & \pscircle*{0.08} & \pscircle*{0.08} & \pscircle*{0.08} & 
\pscircle*{0.08}\\ 
\end{psmatrix}
\ncarc[arcangle=90]{1,2}{1,3}
\ncarc[arcangle=90]{1,4}{1,5}\\
& & \\
\hline
\begin{psmatrix}[colsep=0.35,rowsep=1.2]
\pscircle*{0.08} & \pscircle*{0.08} & \pscircle*{0.08} & \pscircle*{0.08} & 
\pscircle*{0.08} 
\end{psmatrix}
\ncarc[arcangle=90]{1,2}{1,3} & $s_2 s_3$ & \begin{psmatrix}[colsep=0.35,rowsep=1.2]
\pscircle*{0.08} & \pscircle*{0.08} & \pscircle*{0.08} & \pscircle*{0.08} & 
\pscircle*{0.08} 
\end{psmatrix}
\ncarc[arcangle=90]{1,3}{1,4}\\
& & \\
\hline
\begin{psmatrix}[colsep=0.35,rowsep=1.2]
\pscircle*{0.08} & \pscircle*{0.08} & \pscircle*{0.08} & \pscircle*{0.08} & 
\pscircle*{0.08} 
\end{psmatrix}
\ncarc[arcangle=90]{1,3}{1,4} & $s_3 s_2$ & \begin{psmatrix}[colsep=0.35,rowsep=1.2]
\pscircle*{0.08} & \pscircle*{0.08} & \pscircle*{0.08} & \pscircle*{0.08} & 
\pscircle*{0.08}\\ 
\end{psmatrix}
\ncarc[arcangle=90]{1,2}{1,3}\\
& & \\
\hline
\begin{psmatrix}[colsep=0.35,rowsep=1.2]
\pscircle*{0.08} & \pscircle*{0.08} & \pscircle*{0.08} & \pscircle*{0.08} & 
\pscircle*{0.08}\\ 
\end{psmatrix}
\ncarc[arcangle=90]{1,3}{1,4} & $s_3 s_4$ & \begin{psmatrix}[colsep=0.35,rowsep=1.2]
\pscircle*{0.08} & \pscircle*{0.08} & \pscircle*{0.08} & \pscircle*{0.08} & 
\pscircle*{0.08}\\ 
\end{psmatrix}
\ncarc[arcangle=90]{1,4}{1,5}\\
& & \\
\hline
\begin{psmatrix}[colsep=0.35,rowsep=1.2]
\pscircle*{0.08} & \pscircle*{0.08} & \pscircle*{0.08} & \pscircle*{0.08} & 
\pscircle*{0.08}\\ 
\end{psmatrix}
\ncarc[arcangle=90]{1,4}{1,5} & $s_4 s_3$ & \begin{psmatrix}[colsep=0.35,rowsep=1.2]
\pscircle*{0.08} & \pscircle*{0.08} & \pscircle*{0.08} & \pscircle*{0.08} & 
\pscircle*{0.08}\\ 
\end{psmatrix}
\ncarc[arcangle=90]{1,3}{1,4}\\
& & \\
\hline
\begin{psmatrix}[colsep=0.35,rowsep=1.2]
\pscircle*{0.08} & \pscircle*{0.08} & \pscircle*{0.08} & \pscircle*{0.08} & 
\pscircle*{0.08}\\ 
\end{psmatrix}
\ncarc[arcangle=90]{1,1}{1,2} & $s_1 s_2 s_3$ & \begin{psmatrix}[colsep=0.35,rowsep=1.2]
\pscircle*{0.08} & \pscircle*{0.08} & \pscircle*{0.08} & \pscircle*{0.08} & 
\pscircle*{0.08}\\ 
\end{psmatrix}
\ncarc[arcangle=90]{1,3}{1,4}\\
& & \\
\hline
\begin{psmatrix}[colsep=0.35,rowsep=1.2]
\pscircle*{0.08} & \pscircle*{0.08} & \pscircle*{0.08} & \pscircle*{0.08} & 
\pscircle*{0.08}\\ 
\end{psmatrix}
\ncarc[arcangle=90]{1,1}{1,4}
\ncarc[arcangle=90]{1,2}{1,3} & $s_2 s_1 s_3$ & \begin{psmatrix}[colsep=0.35,rowsep=1.2]
\pscircle*{0.08} & \pscircle*{0.08} & \pscircle*{0.08} & \pscircle*{0.08} & 
\pscircle*{0.08}\\ 
\end{psmatrix}
\ncarc[arcangle=90]{1,1}{1,2}
\ncarc[arcangle=90]{1,3}{1,4}\\
& & \\
\hline
\begin{psmatrix}[colsep=0.35,rowsep=1.2]
\pscircle*{0.08} & \pscircle*{0.08} & \pscircle*{0.08} & \pscircle*{0.08} & 
\pscircle*{0.08}\\ 
\end{psmatrix}
\ncarc[arcangle=90]{1,1}{1,2}
\ncarc[arcangle=90]{1,3}{1,4} & $s_3 s_1 s_2$ & \begin{psmatrix}[colsep=0.35,rowsep=1.2]
\pscircle*{0.08} & \pscircle*{0.08} & \pscircle*{0.08} & \pscircle*{0.08} & 
\pscircle*{0.08}\\ 
\end{psmatrix}
\ncarc[arcangle=90]{1,1}{1,4}
\ncarc[arcangle=90]{1,2}{1,3}\\
& & \\
\hline
\begin{psmatrix}[colsep=0.35,rowsep=1.2]
\pscircle*{0.08} & \pscircle*{0.08} & \pscircle*{0.08} & \pscircle*{0.08} & 
\pscircle*{0.08} 
\end{psmatrix}
\ncarc[arcangle=90]{1,3}{1,4} & $s_3 s_2 s_1$ & \begin{psmatrix}[colsep=0.35,rowsep=1.2]
\pscircle*{0.08} & \pscircle*{0.08} & \pscircle*{0.08} & \pscircle*{0.08} & 
\pscircle*{0.08} 
\end{psmatrix}
\ncarc[arcangle=90]{1,1}{1,2}\\
& & \\
\hline
\begin{psmatrix}[colsep=0.35,rowsep=1.2]
\pscircle*{0.08} & \pscircle*{0.08} & \pscircle*{0.08} & \pscircle*{0.08} & 
\pscircle*{0.08} 
\end{psmatrix}
\ncarc[arcangle=90]{1,2}{1,3} & $s_2 s_3 s_4$ & \begin{psmatrix}[colsep=0.35,rowsep=1.2]
\pscircle*{0.08} & \pscircle*{0.08} & \pscircle*{0.08} & \pscircle*{0.08} & 
\pscircle*{0.08}\\ 
\end{psmatrix}
\ncarc[arcangle=90]{1,4}{1,5}\\
& & \\
\hline
\begin{psmatrix}[colsep=0.35,rowsep=1.2]
\pscircle*{0.08} & \pscircle*{0.08} & \pscircle*{0.08} & \pscircle*{0.08} & 
\pscircle*{0.08}\\ 
\end{psmatrix}
\ncarc[arcangle=90]{1,2}{1,5}
\ncarc[arcangle=90]{1,3}{1,4} & $s_3 s_2 s_4$ & \begin{psmatrix}[colsep=0.35,rowsep=1.2]
\pscircle*{0.08} & \pscircle*{0.08} & \pscircle*{0.08} & \pscircle*{0.08} & 
\pscircle*{0.08}\\ 
\end{psmatrix}
\ncarc[arcangle=90]{1,2}{1,3}
\ncarc[arcangle=90]{1,4}{1,5}\\
& & \\
\hline
\begin{psmatrix}[colsep=0.35,rowsep=1.2]
\pscircle*{0.08} & \pscircle*{0.08} & \pscircle*{0.08} & \pscircle*{0.08} & 
\pscircle*{0.08}\\ 
\end{psmatrix}
\ncarc[arcangle=90]{1,2}{1,3}
\ncarc[arcangle=90]{1,4}{1,5} & $s_4 s_2 s_3$ & \begin{psmatrix}[colsep=0.35,rowsep=1.2]
\pscircle*{0.08} & \pscircle*{0.08} & \pscircle*{0.08} & \pscircle*{0.08} & 
\pscircle*{0.08}\\ 
\end{psmatrix}
\ncarc[arcangle=90]{1,2}{1,5}
\ncarc[arcangle=90]{1,3}{1,4}\\
& & \\
\hline
\end{tabular}
&
\begin{tabular}{|c|c|c|}
\hline
\begin{psmatrix}[colsep=0.35,rowsep=1.2]
\pscircle*{0.08} & \pscircle*{0.08} & \pscircle*{0.08} & \pscircle*{0.08} & 
\pscircle*{0.08} 
\end{psmatrix}
\ncarc[arcangle=90]{1,4}{1,5} & $s_4 s_3 s_2$ & \begin{psmatrix}[colsep=0.35,rowsep=1.2]
\pscircle*{0.08} & \pscircle*{0.08} & \pscircle*{0.08} & \pscircle*{0.08} & 
\pscircle*{0.08} 
\end{psmatrix}
\ncarc[arcangle=90]{1,2}{1,3}\\
& & \\
\hline
\begin{psmatrix}[colsep=0.35,rowsep=1.2]
\pscircle*{0.08} & \pscircle*{0.08} & \pscircle*{0.08} & \pscircle*{0.08} & 
\pscircle*{0.08}\\ 
\end{psmatrix}
\ncarc[arcangle=90]{1,1}{1,2}
\ncarc[arcangle=90]{1,4}{1,5} & $s_1 s_2 s_4$ & \begin{psmatrix}[colsep=0.35,rowsep=1.2]
\pscircle*{0.08} & \pscircle*{0.08} & \pscircle*{0.08} & \pscircle*{0.08} & 
\pscircle*{0.08}\\ 
\end{psmatrix}
\ncarc[arcangle=90]{1,2}{1,3}
\ncarc[arcangle=90]{1,4}{1,5}\\
& & \\
\hline
\begin{psmatrix}[colsep=0.35,rowsep=3]
\pscircle*{0.08} & \pscircle*{0.08} & \pscircle*{0.08} & \pscircle*{0.08} & 
\pscircle*{0.08}\\ 
\end{psmatrix}
\ncarc[arcangle=90]{1,2}{1,3}
\ncarc[arcangle=90]{1,4}{1,5} & $s_2 s_1 s_4$ & \begin{psmatrix}[colsep=0.35,rowsep=1.2]
\pscircle*{0.08} & \pscircle*{0.08} & \pscircle*{0.08} & \pscircle*{0.08} & 
\pscircle*{0.08}\\ 
\end{psmatrix}
\ncarc[arcangle=90]{1,1}{1,2}
\ncarc[arcangle=90]{1,4}{1,5}\\
& & \\
\hline
\begin{psmatrix}[colsep=0.35=0.5,rowsep=1.2]
\pscircle*{0.08} & \pscircle*{0.08} & \pscircle*{0.08} & \pscircle*{0.08} & 
\pscircle*{0.08}\\ 
\end{psmatrix}
\ncarc[arcangle=90]{1,1}{1,2}
\ncarc[arcangle=90]{1,3}{1,4} & $s_1 s_3 s_4$ & \begin{psmatrix}[colsep=0.35,rowsep=1.2]
\pscircle*{0.08} & \pscircle*{0.08} & \pscircle*{0.08} & \pscircle*{0.08} & 
\pscircle*{0.08}\\ 
\end{psmatrix}
\ncarc[arcangle=90]{1,1}{1,2}
\ncarc[arcangle=90]{1,4}{1,5}\\
& & \\
\hline
\begin{psmatrix}[colsep=0.35,rowsep=1.2]
\pscircle*{0.08} & \pscircle*{0.08} & \pscircle*{0.08} & \pscircle*{0.08} & 
\pscircle*{0.08}\\ 
\end{psmatrix}
\ncarc[arcangle=90]{1,1}{1,2}
\ncarc[arcangle=90]{1,4}{1,5} & $s_1 s_4 s_3$ & \begin{psmatrix}[colsep=0.35,rowsep=1.2]
\pscircle*{0.08} & \pscircle*{0.08} & \pscircle*{0.08} & \pscircle*{0.08} & 
\pscircle*{0.08}\\ 
\end{psmatrix}
\ncarc[arcangle=90]{1,1}{1,2}
\ncarc[arcangle=90]{1,3}{1,4}\\
& & \\
\hline
\begin{psmatrix}[colsep=0.35,rowsep=1.2]
\pscircle*{0.08} & \pscircle*{0.08} & \pscircle*{0.08} & \pscircle*{0.08} & 
\pscircle*{0.08} 
\end{psmatrix}
\ncarc[arcangle=90]{1,1}{1,2} & $s_1 s_2 s_3 s_4$ & \begin{psmatrix}[colsep=0.35,rowsep=1.2]
\pscircle*{0.08} & \pscircle*{0.08} & \pscircle*{0.08} & \pscircle*{0.08} & 
\pscircle*{0.08} 
\end{psmatrix}
\ncarc[arcangle=90]{1,4}{1,5}\\
& & \\
\hline
\begin{psmatrix}[colsep=0.35,rowsep=1.2]
\pscircle*{0.08} & \pscircle*{0.08} & \pscircle*{0.08} & \pscircle*{0.08} & 
\pscircle*{0.08} 
\end{psmatrix}
\ncarc[arcangle=90]{1,1}{1,4}
\ncarc[arcangle=90]{1,2}{1,3} & $s_2 s_1 s_3 s_4$ & \begin{psmatrix}[colsep=0.35,rowsep=1.2]
\pscircle*{0.08} & \pscircle*{0.08} & \pscircle*{0.08} & \pscircle*{0.08} & 
\pscircle*{0.08}\\ 
\end{psmatrix}
\ncarc[arcangle=90]{1,1}{1,2}
\ncarc[arcangle=90]{1,4}{1,5}\\
& & \\
\hline
\begin{psmatrix}[colsep=0.35,rowsep=1.2]
\pscircle*{0.08} & \pscircle*{0.08} & \pscircle*{0.08} & \pscircle*{0.08} & 
\pscircle*{0.08}\\ 
\end{psmatrix}
\ncarc[arcangle=90]{1,1}{1,2}
\ncarc[arcangle=90]{1,3}{1,4} & $s_1 s_3 s_2 s_4$ & \begin{psmatrix}[colsep=0.35,rowsep=1.2]
\pscircle*{0.08} & \pscircle*{0.08} & \pscircle*{0.08} & \pscircle*{0.08} & 
\pscircle*{0.08}\\ 
\end{psmatrix}
\ncarc[arcangle=90]{1,2}{1,3}
\ncarc[arcangle=90]{1,4}{1,5}\\
& & \\
\hline
\begin{psmatrix}[colsep=0.35,rowsep=1.2]
\pscircle*{0.08} & \pscircle*{0.08} & \pscircle*{0.08} & \pscircle*{0.08} & 
\pscircle*{0.08}\\ 
\end{psmatrix}
\ncarc[arcangle=90]{1,1}{1,2}
\ncarc[arcangle=90]{1,4}{1,5} & $s_1 s_2 s_4 s_3$ & \begin{psmatrix}[colsep=0.35,rowsep=1.2]
\pscircle*{0.08} & \pscircle*{0.08} & \pscircle*{0.08} & \pscircle*{0.08} & 
\pscircle*{0.08}\\ 
\end{psmatrix}
\ncarc[arcangle=90]{1,2}{1,5}
\ncarc[arcangle=90]{1,3}{1,4}\\
& & \\
\hline
\begin{psmatrix}[colsep=0.35,rowsep=1.2]
\pscircle*{0.08} & \pscircle*{0.08} & \pscircle*{0.08} & \pscircle*{0.08} & 
\pscircle*{0.08}\\ 
\end{psmatrix}
\ncarc[arcangle=90]{1,1}{1,2}
\ncarc[arcangle=90]{1,4}{1,5} & $s_1 s_4 s_3 s_2$ & \begin{psmatrix}[colsep=0.35,rowsep=1.2]
\pscircle*{0.08} & \pscircle*{0.08} & \pscircle*{0.08} & \pscircle*{0.08} & 
\pscircle*{0.08}\\ 
\end{psmatrix}
\ncarc[arcangle=90]{1,1}{1,4}
\ncarc[arcangle=90]{1,2}{1,3}\\
& & \\
\hline
\begin{psmatrix}[colsep=0.35,rowsep=1.2]
\pscircle*{0.08} & \pscircle*{0.08} & \pscircle*{0.08} & \pscircle*{0.08} & 
\pscircle*{0.08} 
\end{psmatrix}
\ncarc[arcangle=90]{1,2}{1,3}
\ncarc[arcangle=90]{1,4}{1,5} & $s_2 s_1 s_4 s_3$ & \begin{psmatrix}[colsep=0.35,rowsep=1.2]
\pscircle*{0.08} & \pscircle*{0.08} & \pscircle*{0.08} & \pscircle*{0.08} & 
\pscircle*{0.08} 
\end{psmatrix}
\ncarc[arcangle=90]{1,1}{1,2}
\ncarc[arcangle=90]{1,3}{1,4}\\
& & \\
\hline
\begin{psmatrix}[colsep=0.35,rowsep=1.2]
\pscircle*{0.08} & \pscircle*{0.08} & \pscircle*{0.08} & \pscircle*{0.08} & 
\pscircle*{0.08} 
\end{psmatrix}
\ncarc[arcangle=90]{1,2}{1,5}
\ncarc[arcangle=90]{1,3}{1,4} & $s_3 s_4 s_2 s_1$ & \begin{psmatrix}[colsep=0.35,rowsep=1.2]
\pscircle*{0.08} & \pscircle*{0.08} & \pscircle*{0.08} & \pscircle*{0.08} & 
\pscircle*{0.08}\\ 
\end{psmatrix}
\ncarc[arcangle=90]{1,1}{1,2}
\ncarc[arcangle=90]{1,4}{1,5}\\
& & \\
\hline
\begin{psmatrix}[colsep=0.35,rowsep=1.2]
\pscircle*{0.08} & \pscircle*{0.08} & \pscircle*{0.08} & \pscircle*{0.08} & 
\pscircle*{0.08}\\ 
\end{psmatrix}
\ncarc[arcangle=90]{1,4}{1,5} & $s_4 s_3 s_2 s_1$ & \begin{psmatrix}[colsep=0.35,rowsep=1.2]
\pscircle*{0.08} & \pscircle*{0.08} & \pscircle*{0.08} & \pscircle*{0.08} & 
\pscircle*{0.08}\\ 
\end{psmatrix}
\ncarc[arcangle=90]{1,1}{1,2}\\
& & \\
\hline
\begin{psmatrix}[colsep=0.35,rowsep=1.2]
\pscircle*{0.08} & \pscircle*{0.08} & \pscircle*{0.08} & \pscircle*{0.08} & 
\pscircle*{0.08}\\ 
\end{psmatrix}
\ncarc[arcangle=90]{1,1}{1,4}
\ncarc[arcangle=90]{1,2}{1,3} & $s_2 s_1 s_3 s_2$ & \begin{psmatrix}[colsep=0.35,rowsep=1.2]
\pscircle*{0.08} & \pscircle*{0.08} & \pscircle*{0.08} & \pscircle*{0.08} & 
\pscircle*{0.08}\\ 
\end{psmatrix}
\ncarc[arcangle=90]{1,1}{1,4}
\ncarc[arcangle=90]{1,2}{1,3}\\
& & \\
\hline
\begin{psmatrix}[colsep=0.35,rowsep=1.2]
\pscircle*{0.08} & \pscircle*{0.08} & \pscircle*{0.08} & \pscircle*{0.08} & 
\pscircle*{0.08}\\ 
\end{psmatrix}
\ncarc[arcangle=90]{1,2}{1,5}
\ncarc[arcangle=90]{1,3}{1,4} & $s_3 s_2 s_4 s_3$ & \begin{psmatrix}[colsep=0.35,rowsep=1.2]
\pscircle*{0.08} & \pscircle*{0.08} & \pscircle*{0.08} & \pscircle*{0.08} & 
\pscircle*{0.08}\\ 
\end{psmatrix}
\ncarc[arcangle=90]{1,2}{1,5}
\ncarc[arcangle=90]{1,3}{1,4}\\
& & \\
\hline
\begin{psmatrix}[colsep=0.35,rowsep=1.2]
\pscircle*{0.08} & \pscircle*{0.08} & \pscircle*{0.08} & \pscircle*{0.08} & 
\pscircle*{0.08}\\ 
\end{psmatrix}
\ncarc[arcangle=90]{1,2}{1,3}
\ncarc[arcangle=90]{1,4}{1,5} & $s_4 s_2 s_3 s_1 s_2$ & \begin{psmatrix}[colsep=0.35,rowsep=1.2]
\pscircle*{0.08} & \pscircle*{0.08} & \pscircle*{0.08} & \pscircle*{0.08} & 
\pscircle*{0.08}\\ 
\end{psmatrix}
\ncarc[arcangle=90]{1,1}{1,4}
\ncarc[arcangle=90]{1,2}{1,3}\\
& & \\
\hline
\begin{psmatrix}[colsep=0.35,rowsep=1.2]
\pscircle*{0.08} & \pscircle*{0.08} & \pscircle*{0.08} & \pscircle*{0.08} & 
\pscircle*{0.08}\\ 
\end{psmatrix}
\ncarc[arcangle=90]{1,2}{1,5}
\ncarc[arcangle=90]{1,3}{1,4} & $s_3 s_4 s_2 s_3 s_1$ & \begin{psmatrix}[colsep=0.35,rowsep=1.2]
\pscircle*{0.08} & \pscircle*{0.08} & \pscircle*{0.08} & \pscircle*{0.08} & 
\pscircle*{0.08}\\ 
\end{psmatrix}
\ncarc[arcangle=90]{1,1}{1,2}
\ncarc[arcangle=90]{1,3}{1,4}\\
& & \\
\hline
\begin{psmatrix}[colsep=0.35,rowsep=1.2]
\pscircle*{0.08} & \pscircle*{0.08} & \pscircle*{0.08} & \pscircle*{0.08} & 
\pscircle*{0.08} 
\end{psmatrix}
\ncarc[arcangle=90]{1,1}{1,4}
\ncarc[arcangle=90]{1,2}{1,3} & $s_2 s_1 s_3 s_2 s_4$ & \begin{psmatrix}[colsep=0.35,rowsep=1.2]
\pscircle*{0.08} & \pscircle*{0.08} & \pscircle*{0.08} & \pscircle*{0.08} & 
\pscircle*{0.08} 
\end{psmatrix}
\ncarc[arcangle=90]{1,2}{1,3}
\ncarc[arcangle=90]{1,4}{1,5}\\
& & \\
\hline
\begin{psmatrix}[colsep=0.35,rowsep=1.2]
\pscircle*{0.08} & \pscircle*{0.08} & \pscircle*{0.08} & \pscircle*{0.08} & 
\pscircle*{0.08} 
\end{psmatrix}
\ncarc[arcangle=90]{1,1}{1,2}
\ncarc[arcangle=90]{1,3}{1,4} & $s_1 s_3 s_2 s_4 s_3$ & \begin{psmatrix}[colsep=0.35,rowsep=1.2]
\pscircle*{0.08} & \pscircle*{0.08} & \pscircle*{0.08} & \pscircle*{0.08} & 
\pscircle*{0.08}\\ 
\end{psmatrix}
\ncarc[arcangle=90]{1,2}{1,5}
\ncarc[arcangle=90]{1,3}{1,4}\\
& & \\
\hline
\begin{psmatrix}[colsep=0.35,rowsep=1.2]
\pscircle*{0.08} & \pscircle*{0.08} & \pscircle*{0.08} & \pscircle*{0.08} & 
\pscircle*{0.08}\\ 
\end{psmatrix}
\ncarc[arcangle=90]{1,2}{1,5}
\ncarc[arcangle=90]{1,3}{1,4} & $s_2 s_1 s_3 s_2 s_4 s_3$ & \begin{psmatrix}[colsep=0.35,rowsep=1.2]
\pscircle*{0.08} & \pscircle*{0.08} & \pscircle*{0.08} & \pscircle*{0.08} & 
\pscircle*{0.08}\\ 
\end{psmatrix}
\ncarc[arcangle=90]{1,2}{1,5}
\ncarc[arcangle=90]{1,3}{1,4}\\
& & \\
\hline
\begin{psmatrix}[colsep=0.35,rowsep=1.2]
\pscircle*{0.08} & \pscircle*{0.08} & \pscircle*{0.08} & \pscircle*{0.08} & 
\pscircle*{0.08}\\ 
\end{psmatrix}
\ncarc[arcangle=90]{1,2}{1,5}
\ncarc[arcangle=90]{1,3}{1,4} & $s_3 s_4 s_2 s_3 s_1 s_2$ & \begin{psmatrix}[colsep=0.35,rowsep=1.2]
\pscircle*{0.08} & \pscircle*{0.08} & \pscircle*{0.08} & \pscircle*{0.08} & 
\pscircle*{0.08}\\ 
\end{psmatrix}
\ncarc[arcangle=90]{1,1}{1,4}
\ncarc[arcangle=90]{1,2}{1,3}\\
& & \\
\hline
\end{tabular}
\end{tabular}
\caption{Left and right dense sets of reflections for any fully commutative element in type $A_4$.}
\end{center}
\end{figure}
\begin{prop}\label{prop:digne}
Let $w\in\mathcal{W}_c$ and suppose $w=s_{i_1}\cdots s_{i_k}$ is a reduced expression. Then
$$T(i_1\cdots i_k)=Q(i_1\cdots i_k).$$
\end{prop} 
\begin{proof}
We use induction on $k$ ; if $k=1$, then $T(i_1)=\{s_{i_1}\}$ and the dense set at the top of the diagram corresponding to $b_{i_1}$ contains only the reflection $s_{i_1}$. We suppose that the result holds for a sequence of length at most $k-1$. By induction, $T(i_2\cdots i_k)=Q(i_2\cdots i_k)$ and it suffices to show that the same three rules given in lemma \ref{lem:reflexions-recurrence} hold when passing from $Q(i_2\cdots i_k)$ to $Q(i_1\cdots i_k)$. If $s_{i_1}$ commutes with any reflection in $Q(i_2\cdots i_k)$, then the dense set at the top of $b_w$ is $Q(i_2\cdots i_k)\cup\{s_{i_1}\}$. If $s_{i_1}$ commutes with exactly one reflection $t$ in $Q(i_2\cdots i_k)$ then $t$ will become a line from the top to the bottom of the diagram associated to $b_w$ when collapsing the diagrams for $b_{i_1}$ and $b_{s_{i_1}w}$ and hence $t$ disappears from $Q(i_2\cdots i_k)$, $s_{i_1}$ is added and all other reflections become unchanged, hence $Q(i_1\cdots i_k)=(Q(i_2\cdots i_k)\backslash t)\cup \{s_{i_1}\}$. If $s_{i_1}$ commutes with two distinct reflections $(j_1, i_1),(i_1+1, j_2)\in Q(i_2\cdots i_k)$ with $j_1<i_1$, $i_1+1<j_2$, one sees by drawing the situation that when concatenating the diagram associated to $b_{i_1}$ to the one associated to $b_{s_{i_2}\cdots s_{i_k}}$, no line from the top to the bottom of the diagram corresponding to $b_w$ is added, that the simple reflection $s_{i_1}$ which lies at the bottom of the diagram corresponding to $b_{i_1}$ will joint the index $i_1$ to the index $i_1+1$, removing the above two reflections $(j_1, i_1),(i_1+1, j_2)$ to replace them by $(j_1, j_2)$, that of course the simple reflection $s_{i_1}$ coming from the top of the diagram of $b_{i_1}$ is added and that all other reflections in $Q(i_2\cdots i_k)$ stay unchanged, hence $$Q(i_1\cdots i_k)=\big(Q(i_2\cdots i_k)\backslash\{(j_1, i_1),(i_1+1,j_2)\}\big)\cup\{s_{i_1}, (j_1, j_2)\}.$$
We deduce from lemma \ref{lem:reflexions-recurrence} that $T(i_1\cdots i_k)=Q(i_1\cdots i_k)$.   
\end{proof}
\begin{cor}\label{cor:noniso}
The bimodules $B_w$ for $w\in\mathcal{W}_c$ are pairwise non-isomorphic in $\bar{R}-\mathrm{mod}-\bar{R}$ (hence in $\bar{R}-\mathrm{mod}_{\mathbb{Z}-}\bar{R}$).
\end{cor}
\begin{proof}
If $w\in\mathcal{W}_c$ with $s_{i_1}\cdots s_{i_k}$ a reduced expression, then the planar diagram corresponding to the element $b_w\in\mathrm{TL}_n$ is entirely determined by the two dense sets obtained by removing the lines going from the top to the bottom of the diagram, that is the pair $(Q(i_1\cdots i_k), Q(i_k\cdots i_1))$, since the lines in the diagram must be noncrossing. Hence two distinct fully commutative elements $w, w'\in\mathcal{W}$ will have distincts such pairs. Using proposition \ref{prop:digne}, the corresponding fully commutative bimodules $B_w$ and $B_{w'}$ will then have distinct left annihilators or distinct right annihilators, hence will be non-isomorphic as $(\bar{R}, \bar{R})$-bimodules.
\end{proof}
\subsection{Indecomposability of fully commutative bimodules}

The next step is to prove indecomposability of $\ast$-products of $B_i$ bimodules corresponding to elements of the Kazhdan-Lusztig basis of the Temperley-Lieb algebra, that is, fully commutative bimodules $B_w$. Any element $b_w\in\mathrm{TL}_n$ with $w\in\mathcal{W}_c$ can be written as a product
$$(b_{{i_k}} b_{{i_k-1}}\cdots b_{{j_k}})(b_{{i_{k-1}}} b_{{i_{k-1}-1}}\cdots b_{{j_{k-1}}})\cdots (b_{{i_1}} b_{{i_1-1}}\cdots b_{{j_1}})$$ with all indices in $\{1, \dots, n\}$ and $i_k<i_{k-1}<\dots <i_1$, $j_{k}<j_{k-1}<\dots <j_1$ and $j_m\leq i_m$ for each $m=1,\dots,k$ (see \cite{KT}, §5.7; we have reversed the indices $1, \dots, k$ since it will be more convenient for the inductions we will use later). 

Since the bimodules $B_i$ satisfy the Temperley-Lieb relations any fully commutative bimodule can written in the form
$$(B_{{i_k}} B_{i_k-1}\cdots B_{{j_k}})(B_{{i_{k-1}}} B_{{i_{k-1}-1}}\cdots B_{{j_{k-1}}})\cdots (B_{{i_1}} B_{{i_1-1}}\cdots B_{{j_1}}).$$

\begin{defn} We say that such a fully commutative bimodule is \textbf{associated} to the corresponding sequence $$i_k\cdots j_k i_{k-1}\cdots j_{k-1}\cdots i_1\cdots j_1.$$ The integer $k$ is the \textbf{rank} of the sequence. A fully commutative bimodule is \textbf{intertwined} if for each $1< m\leq k$, the set $[i_m, j_m]$ contains both the indices $i_1-2(m-1)$ and $i_1-2(m-1)+1$. 
\end{defn}
\begin{exple}
In case $n\geq 9$, the bimodule associated to the sequence 
$$(1) (432) (654) (7) (98)$$
is not intertwined. Bimodules associated to the sequences
$$(3\red21\black) (\red43\black) (7\red65\black4) (\red87\black6) (9),~(\red21\black) (\red43\black) (\red65\black) (\red87\black) (98),~(543\red21\black) (65\red43\black) (7\red65\black) (\red87\black) (9)$$\noindent are intertwined ; here $i_1=9$ and the indices of the form $i_1-2(\ell-1)$ and $i_1-2(\ell-1)+1$ from the definition are drawn in red. As an exercise the reader can compute the dense sets of reflections characterizing the varieties of the left and right annihilators. 
\end{exple} 

\begin{lem}\label{lem:indec1}
Let $B$ be a fully commutative bimodule. If $B$ is associated to a sequence of rank $1$, then $B$ is indecomposable (as graded bimodule).
\end{lem}
\begin{proof}
We write $i(i-1)\cdots j$ for the sequence associated to our bimodule, $i-j\geq 0$. One has
$$B=B_{i}\ast B_{i-1}\ast\cdots\ast B_{j}.$$
One has $W_{m(m-1)\cdots j}=V_m$ and $W_{m(m +1)\cdots i}=V_{m}$ for each $m\in[j,i]$ thanks to lemma \ref{lem:suitecroissante}. As a consequence with any choice of brackets for computing the above product one gets that $B$ is isomorphic to 
$$B_i\otimes_{R_i} R_{i, i-1}\otimes_{R_{i-1}}B_{i-1}\otimes_{R_{i-1}} R_{i-1, i-2}\otimes_{R_{i-2}} B_{i-2} \otimes\cdots\otimes R_{j+1,j}\otimes_{R_j} B_j,$$
\noindent with $R_{m, m-1}=\mathcal{O}(V_m\cap V_{m-1})$ for each $m\in[j+1,i]$ ; if $i=j$ we get $B_i=B_j$. After reduction $B$ is isomorphic to 
$$R_{i}\otimes_i R_{i, i-1}\otimes_{i-1} R_{i-1, i-2}\otimes_{i-2}\cdots\otimes_{j+1} R_{j+1,j}\otimes_j R_j,$$
\noindent where $\otimes_m$ means $\otimes_{\bar{R}^{s_m}}$; if $i=j$ we get $R_i\otimes_i R_i$. Thanks to remark \ref{rmq:noninvariant} one then has $\bar{R}_m^{s_m}\twoheadrightarrow\mathcal{O}(V_m\cap V_{m-1})$ as well as $\bar{R}_{m-1}^{s_{m-1}}\twoheadrightarrow\mathcal{O}(V_m\cap V_{m-1})$ for each $m\in[j+1,i]$. Hence any tensor 
$$a_i\otimes_i a_{i, i-1}\otimes_{i-1}\cdots \otimes_{j+1} a_{j+1, j}\otimes_j a_j\in B$$ is equal to a tensor
$$a\otimes_i 1\otimes_{i-1}\cdots \otimes_{j+1} 1\otimes_j a'\in B.$$ As a consequence $B$ is generated as $(\bar{R},\bar{R})$-bimodule by the degree zero element $1\otimes_i 1\otimes_{i-1}\cdots\otimes_{j+1} 1\otimes_j 1$ which forces indecomposability since the zero degree component of $B$ has dimension $1$. 
\end{proof}
\begin{lem}\label{lem:decomposer}
Consider the bimodule $B$ from the proof of lemma \ref{lem:indec1} written in the form $$R_{i}\otimes_i R_{i, i-1}\otimes_{i-1} R_{i-1, i-2}\otimes_{i-2}\cdots\otimes_{j+1} R_{j+1,j}\otimes_j R_j.$$ Any tensor $a\otimes_i 1\otimes_{i-1}\cdots \otimes_{j+1} 1\otimes_j a'\in B$ where $a\in R_i$, $a'\in R_j$ can be written in the form $$(b\otimes_i 1\otimes_{i-1}\cdots \otimes_{j+1} 1\otimes_j 1)+(b'\otimes_i 1\otimes_{i-1}\cdots \otimes_{j+1} 1\otimes_j f_j),$$ where $b, b'\in R_i$.
\end{lem}
\begin{proof}
It suffices to decompose $a'=r+r'f_j$ with $r, r'\in R_j^{s_j}$ and move $r, r'$ to the left using the fact that $\bar{R}_m^{s_m}\twoheadrightarrow\mathcal{O}(V_m\cap V_{m-1})$ as well as $\bar{R}_{m-1}^{s_{m-1}}\twoheadrightarrow\mathcal{O}(V_m\cap V_{m-1})$ for each $m\in[j+1,i]$.
\end{proof} 
\begin{nota}
Let $i_k\cdots j_k\cdots i_1\cdots j_1$ be a sequence defining a fully commutative bimodule $B$. For each $m\in [1, k]$, we write $B(m)$ for the bimodule associated to the subsequence $i_m\cdots j_m$. In particular we have
$$B\cong B(k)\ast B(k-1)\ast\cdots\ast B(1).$$ 
\end{nota}
\begin{prop}\label{prop:entrelace}
Let $B$ be an intertwined bimodule associated to the sequence 
$\mathrm{seq}=i_k\cdots j_k i_{k-1} \cdots j_{k-1} \cdots i_1\cdots j_1$. 
\begin{enumerate}
\item One has the equality $\mathrm{supp}(T(\mathrm{seq}))=[i_1-2(k-1), i_1+1]$. Moreover, the set $T(\mathrm{seq})$ contains the reflection $(i_1-2(k-1), i_1+1)$ (in other words, it has a single block). 
\item The bimodule $B$ is indecomposable.
\end{enumerate}
\end{prop}
\begin{proof}


The first claim is easily shown by induction on $k$. If $k=1$, one has $\mathrm{seq}=i_1\cdots j_1$ and $T(\mathrm{seq})=\{s_{i_1}\}$ (see lemma \ref{lem:suitecroissante}) whose support is $\{i_1, i_1+1\}$. 

Now suppose the result holds for any sequence of rank at most $k-1$ and consider the case where the sequence has rank $k$. If $W=W_{i_{k-1}\cdots j_1}$, then by induction $\mathrm{supp}(T_W)=[i_1-2(k-2), i_1+1]$ and $T_W$ contains the reflection $(i_1-2(k-2), i_1+1)$. Now consider the subsequence $i_k\cdots j_k$ of $\mathrm{seq}$, which is equal to the concatenation of the decreasing sequences $\mathrm{seq_1}=i_k\cdots (i_1-2(k-1)+1)$ and $\mathrm{seq_2}=(i_1-2(k-1))\cdots j_k$ (since the bimodule is intertwined). Any reflection $s_j$ with $j$ in $\mathrm{seq}_2$ commutes with any reflection in $T_W$ hence one gets using lemma \ref{lem:reflexions-recurrence} that $T_{\mathrm{seq_2}\cdot W}=T_W\cup\{s_{i_1-2(k-1)}\}$. We now study the effect of applying $\mathrm{seq_1}$ to $\mathrm{seq_2}\cdot W$. Using again lemma \ref{lem:reflexions-recurrence}, applying the first index on the right of $\mathrm{seq}_1$, that is $(i_1-2(k-1)+1)$, replaces the reflexions $s_{i_1-2(k-1)}$ and $(i_1-2(k-2), i_1+1)$ in $T_W\cup\{s_{i_1-2(k-1)}\}$ by $s_{i_1-2(k-1)+1}$ and $(i_1-2(k-1), i_1+1)$ and applying the following indices only removes and adds reflexions supported in $[i_1-2(k-1)+1, i_1]$, showing that $T_{\mathrm{seq}\cdot W}$ has support equal to $[i_1-2(k-1), i_1+1]$ and contains the reflection $(i_1-2(k-1), i_1+1)$. 

To show indecomposability of $B$, we first compute the $\ast$-product occuring in the bimodules $B(m)$ associated to each decreasing subsequence $\mathrm{seq}_m=i_m\cdots j_m$ of our sequence. These ones occur to be indecomposable thanks to lemma \ref{lem:indec1} and we will write them as in the proof of this lemma in the form 
$$R_{i_m}\otimes_{i_m} R_{i_m, i_m-1}\otimes_{i_m-1} R_{i_m-1, i_m-2}\otimes\cdots\otimes R_{j_m+1,j_m}\otimes_{j_m} R_{j_m}.$$
We will abuse notation and write $B(m)$ for the above isomorphic bimodule. It remains to make a choice of brackets for computing the product $B(k)\ast B(k-1)\ast\cdots\ast B(2)\ast B(1)$. We will compute the product "from the right", i.e.,
$$B(k)\ast (B(k-1)\ast(\cdots\ast (B(3)\ast (B(2)\ast B(1)))\cdots)).$$
Thanks to theorem \ref{thm:asso} together with the first part of the proposition, one has that for $\ell\leq k$, the left annihilator of the intertwined bimodule	$$B(\ell-1)\ast (B(\ell-2)\ast(\cdots\ast (B(3)\ast (B(2)\ast B(1)))\cdots))$$ is equal to the ideal of functions vanishing on $\bigcap_{t\in Q_\ell} V_t$ where $Q_\ell\subset T$ is a dense subset satisfying $\mathrm{supp}(Q_\ell)=[i_1-2(\ell-2), i_1+1]$ and containing the reflection $(i_1-2(\ell-2), i_1+1)$. The right annihilator of $B(\ell)$ is equal to $I(V_{j_{\ell}})$. Since the bimodule $B$ is intertwined one has that $j_{\ell}\leq i_1-2(\ell-1)=i_1-2(\ell-2)-2$ and in particular, $s_{j_{\ell}}$ commutes with any reflection in $Q_\ell$. Set $X_{\ell}=\bigcap_{t\in Q_\ell} V_t$, $W_{\ell}:=V_{s_{j_{\ell}}}\cap X_{\ell}$ for $\ell>1$ and $W_1=V_{j_1}$. One has that $W_\ell$ is $s_{j_{\ell}}$-invariant and hence we can decompose 
\begin{equation}\label{eq:dec}
\mathcal{O}(W_{\ell})=\mathcal{O}(W_{\ell})^{s_{j_{\ell}} }\oplus \mathcal{O}(W_{\ell})^{s_{j_{\ell}} } f_{j_{\ell}}|_{W_{\ell}}.
\end{equation}
We will abuse notation and write $f_i$ instead of $f_i|_X$ for the image of $f_i$ in $\mathcal{O}(X)$ where $X\subset Z$ is an algebraic set to avoid using two much indices and since this will make no possible confusion in the next computations. Computing recursively our product with the above choice of brackets we get that our bimodule $B$ is isomorphic to 
$$B(k)\otimes_{R_{j_k}} \mathcal{O}(W_{k})\otimes_{\mathcal{O}(X_k)} B(k-1)\otimes\cdots \otimes B(2)\otimes_{R_{j_2}}\mathcal{O}(W_2)\otimes_{R_{i_1}} B(1).$$
Again we abuse notation and write $B$ for this isomorphic bimodule. We have seen in the proof of lemma \ref{lem:indec1} that the bimodule $B(\ell)$ is indecomposable and generated by the element $1_\ell:=1\otimes_{i_\ell} 1\otimes_{i_{\ell}-1} 1\otimes\cdots\otimes_{j_{\ell}} 1\in B(\ell)$ for each $\ell$. Hence using lemma \ref{lem:decomposer} any tensor in the above tensor product can be written as a sum of two elements of the form
$$a\cdot 1_{k}\otimes_{R_{j_k}} a_k\otimes_{R_{j_{k-1}}} 1_{k-1} \otimes \cdots\otimes 1_2\otimes_{R_{j_2}} a_2\otimes_{R_{j_1}} 1_1\cdot a',$$ the first one with $a'=1$, $a\in\bar{R}$, $a_{\ell}\in \mathcal{O}(W_{\ell})$ and the second one having the same properties but with $a'=f_{j_1}$. Our strategy is the same as in lemma \ref{lem:indec1}: we will show that our bimodule can be generated by the element
$$1_k\otimes_{R_{j_k}} 1 \otimes_{R_{j_{k-1}}} 1_{k-1}\otimes\cdots\otimes 1_2\otimes_{R_{j_2}} 1\otimes_{R_{j_1}} 1_1.$$ In that case, because of the $s_{j_\ell}$-invariance of the variety $W_{\ell}$, we use relation \ref{eq:dec} to move the invariant parts of each $a_k$ to the left in the same way as at the end of the proof of lemma \ref{lem:indec1}: we begin with $a_2$, writing $a_2=r_2+r'_2 f_{j_2}$ where $r_2$ and $r'_2$ are $s_{j_2}$-invariant. But then one has that in $$\mathcal{O}(W_3)\otimes_{\mathcal{O}(X_3)} B(2)\otimes_{R_{j_2}} \mathcal{O}(W_2),$$ $a_3\otimes 1_2\otimes r_2= q\otimes 1_2\otimes 1$ and $a_3\otimes 1_2 \otimes r'_2 f_{j_2}=q'\otimes 1_2\otimes f_{j_2}$ with $q, q'\in \mathcal{O}(W_3)$. In other words a tensor in $B$ of the form
$$a\cdot 1_{k}\otimes_{R_{j_k}} a_k\otimes_{\mathcal{O}(X_k)} 1_{k-1} \otimes \cdots\otimes a_3\otimes_{\mathcal{O}(X_3)} 1_2\otimes_{R_{j_2}} a_2\otimes_{R_{i_1}} 1_1\cdot a'$$
\noindent is equal to a tensor of the form
$$a\cdot 1_{k}\otimes_{R_{j_k}} a_k\otimes_{\mathcal{O}(X_k)} 1_{k-1} \otimes \cdots \otimes(q\otimes_{\mathcal{O}(X_3)} 1_2\otimes_{R_{j_2}} 1+q'\otimes_{\mathcal{O}(X_3)} 1_2\otimes_{R_{j_2}} f_{j_2})\otimes_{R_{i_1}} 1_1\cdot a'.$$

Now one can decompose $q, q'$ and again "move" the $s_{j_3}$-invariant parts to the left, and so on. At the end of the process we get a sum of elements $\sum_i a_i\cdot t_i$ where $t_i$ are tensors in $B$ with $f_{\ell}$ or $1$ in the $\mathcal{O}(W_{\ell})$-component of $B$ and $1$ in any other component. It remains to show that each of these $t_i$ can be written as a sum of elements of the form $b\cdot 1\otimes 1\otimes\cdots\otimes 1\otimes 1\cdot b'$ with $b, b'\in \bar{R}$ to show that the arbitrary tensor in $B$ we began with can be obtained from the tensor $1\otimes 1\otimes\cdots 1\otimes 1\in B$. In fact we will show that we can write any of the $t_i$ as a single tensor of the form $b\cdot 1\otimes 1\otimes\cdots\otimes 1\otimes 1\cdot b'$ with $b=1$ (in other words, all the remaining $f_{\ell}$ in our tensors will be "moved" to the right) and $b'$ beeing equal to a polynomial in $f_i$ for $i\leq i_1$. For this we need the following technical lemma :
\begin{lem}\label{lem:tech}
Let $B(i)$, $W_i$, $X_i$ be as above for each $2\leq i\leq k$ and set $W_1=V_{j_1}$. Let $\ell\in[2,k]$ and suppose $m\leq i_1-2(\ell-1)$. Then the tensor $f_m\otimes 1_{\ell-1} \otimes 1$ in $\mathcal{O}(W_{\ell})\otimes_{\mathcal{O}(X_\ell)} B(\ell-1)\otimes_{R_{j_{\ell-1}}}\mathcal{O}(W_{\ell-1})$ is equal to a tensor of the form 
$1\otimes 1_{\ell-1}\otimes \sum_{j} f_j$ in the same tensor product with all indices $j\leq i_1-2(\ell-2)$. 
\end{lem}  
\begin{proof}
The first case is the case where $m<j_{\ell-1}-1$. In that case $f_m$ is invariant by any reflection $s_{m'}$ with $m'$ an index occuring in the sequence $i_{\ell-1}\cdots j_{\ell-1}$ and hence $f_m\otimes 1_{\ell-1}\otimes 1=1\otimes 1_{\ell-1}\otimes f_m$ since all the tensor products in $B(\ell-1)$ are over various $\bar{R}^{s_{m'}}$ for $m'$ occuring in the sequence $i_{\ell-1}\cdots j_{\ell-1}$. 

The second case is the case where $m=j_{\ell-1}-1 < i_1-2(\ell-2)-1$, then $m$ and $i_{\ell-1}$ are distant: since our bimodule is intertwined $i_{\ell-1}\cdots j_{\ell-1}$ has to contain the index $i_1-2(\ell-2)+1$ if $\ell>2$, which forces $j_{\ell-1}<i_{\ell-1}$ and if $\ell=2$, the condition $m=j_1-1$ forces $i_1>j_1$ since otherwise one would have $m=i_1-1$ contradicting our assumption that $m\leq i_1-2(\ell-1)$. Hence our tensor is equal to the tensor 
$$1\otimes_{\mathcal{O}(X_k)} 1\otimes_{i_{\ell-1}} 1\otimes \cdots f_m\otimes_{j_{\ell-1}} 1\otimes_{R_{j_{\ell-1}}} 1$$
with $f_m$ lying in $\mathcal{O}(V_{j_{\ell-1}+1}\cap V_{j_{\ell-1}})$. But in this ring we have $f_{j_{\ell-1}+1}+f_{j_{\ell-1}}=0$ since $V_{j_{\ell-1}+1}\cap V_{j_{\ell-1}}\subset H$ where $H$ is the reflecting hyperplane of $s_{j_{\ell-1}+1}s_{j_{\ell-1}}s_{j_{\ell-1}+1}$ (see lemma \ref{lem:hyperebene}), hence in $\mathcal{O}(V_{j_{\ell-1}+1}\cap V_{j_{\ell-1}})$ we get
$$f_m=f_{j_{\ell-1}-1}=f_{j_{\ell-1}+1}+f_{j_{\ell-1}}+f_{j_{\ell-1}-1},$$
which is $s_{j_{\ell-1}}$-invariant, hence the sum in the right hand side can be moved to the last component of the tensor product ; but this is a sum of $f_j$ for $j\leq j_{\ell-1}+1=m+2\leq i_1-2(\ell-2)$.
 
The last case is the case where $m\geq j_{\ell-1}$. This forces $m$ to occur as an index of the sequence $i_{\ell-1}\cdots j_{\ell-1}$ and $m+1$, $m+2$ also occur since the bimodule is intertwined and $m\leq i_1-2(\ell-1)$. In that case our tensor $f_m\otimes 1_{\ell-1}\otimes 1$ is equal to a tensor 
$$1\otimes_{\mathcal{O}(X_\ell)} 1\otimes_{i_{\ell-1}} 1\otimes \cdots\otimes_{m+2} f_m\otimes_{m+1}\otimes 1\otimes_{m} 1\otimes\cdots\otimes 1,$$
with $f_m$ lying in $\mathcal{O}(V_{m+1}\cap V_{m+2})$. In that ring one has $f_m=f_m+f_{m+1}+f_{m+2}$ which is $s_{m+1}$-invariant, hence the sum can be moved to the next factor which is $\mathcal{O}(V_{m}\cap V_{m+1})$. But in that ring, one has $f_m+f_{m+1}=0$, hence our tensor is equal to the tensor
$$1\otimes_{\mathcal{O}(X_\ell)} 1\otimes_{i_{\ell-1}} 1\otimes \cdots\otimes 1\otimes_{m+1}\otimes f_{m+2}\otimes_{m} 1\otimes_{m-1} 1\otimes\cdots\otimes 1,$$ 
and the $f_{m+2}$ can be moved to the right since it is invariant under the operation of all $s_j$ for $j\leq m$. Hence the tensor is equal to 
$$1\otimes_{\mathcal{O}(X_\ell)} 1\otimes_{i_{\ell-1}} 1\otimes \cdots\otimes 1\otimes_{m+1}\otimes 1\otimes_{m} 1\otimes_{m-1} 1\otimes\cdots \otimes1\otimes f_{m+2},$$ and $m+2\leq i_1-2(\ell-2)$, which concludes.
\end{proof}
\noindent \textit{End of the proof of the proposition}. Using the above lemma we can move our $f_{\ell}$'s in the $\mathcal{O}(W_\ell)$ components of our bimodule $B$ to the right inductively, begining from the left with $\ell=k$ by moving $f_\ell$ to the right in the $\mathcal{O}(W_{{\ell}-1})$ component and so on.
\end{proof}
We now consider the indecomposibility of a slightly more general family of bimodules. 
\begin{defn}\label{def:gi}
A fully commutative bimodule associated to a sequence $$i_k\cdots j_k\cdots i_1\cdots j_1$$ will be called a \textbf{generalized intertwined bimodule} if the following condition holds : each set $\{i_\ell,\dots, j_\ell\}$ contains a nonempty subset $S_\ell$ of cardinal at most two such that the following inductive condition is satisfied : $S_1=\{i_1\}$, and if $n(\ell)$ is the lowest index in $S_\ell$, then the set $\{i_{\ell+1},\dots, j_{\ell+1}\}$ contains the index $n(\ell)-1$ and we put
$$
S_{\ell+1} = \left\{
    \begin{array}{ll}
        \{n(\ell)-1\} & \mbox{if } n(\ell)-1=j_{\ell+1} \\
        \{n(\ell)-2, n(\ell)-1\} & \mbox{otherwise.}
    \end{array}
\right.
$$
The union of the sets $S_\ell$ for $\ell=1,\dots,k$ is called the set of \textbf{intertwining indices} of the corresponding sequence or bimodule.
\end{defn}
\begin{exple}
In case $n\geq 9$, the bimodules associated to the following sequences 
$$(\red1\black)(\red32\black)(\red4\black)(7\red65\black)(\red87\black)(\red9\black),~(\red1\black)(\red2\black)(\red43\black)(7\red65\black4)(\red87\black65)(\red9\black876),~(\red87\black)(\red9\black)$$are generalized intertwined bimodules ; the indices belonging to the set $S_\ell$ are written in red. The bimodules associated to the sequences 
$$(1)(32)(65)(87)(9),~(7)(98)$$
\noindent are not generalized intertwined bimodules.
\end{exple}
The following technical result will allow us to use the same kind of arguments as for intertwined bimodules to show indecomposability ; for this, we order the set of simple reflections by setting $s_i<s_j$ if and only if $i<j$, for $i, j\in [1,n]$. 
\begin{lem}\label{lem:indiceminimal}
Let $i_k\cdots j_k\cdots i_1\cdots j_1$ be a sequence defining a generalized intertwined bimodule with corresponding variety $W\in\mathcal{V}_n$. Then 
\begin{enumerate}
	\item The smallest index in $\mathrm{supp}(T_W)$ is equal to $n(k)$ as in definition \ref{def:gi},
	\item The lowest simple reflection occuring in $T_W$ is $s_{i_k}$.  
\end{enumerate} 
\end{lem}
\begin{proof}
We use induction on $k$ ; if $k=1$, the result is trivially true since $T_W=\{s_{i_1}\}$ and $n(1)=i_1$. Now suppose $k>1$. By induction the smallest index occuring in $T_{W'}$ where $W'$ is associated to the sequence $i_{k-1}\cdots j_{k-1}\cdots i_1\cdots j_1$ is $n(k-1)$ (in particular there exists $j>n(k-1)$ such that $(n(k-1),j)\in T_{W'}$) and the lowest simple reflection occuring in $T_{W'}$ is $s_{i_{k-1}}$. 

First consider the case $|S_k|=1$, we then have $n(k)=j_k=n(k-1)-1$. We get $T_{j_k\cdot W'}=(T_{W'}\backslash\{(n(k-1),j)\})\cup\{s_{n(k)}\}$. If $(n(k-1), j)$ is simple, then $n(k-1)=i_{k-1}$ and $i_k=j_k$ since $j_k=i_{k-1}-1$ and $j_k\leq i_k<i_{k-1}$ ; in that case we are done. Otherwise, the two first blocks (from the left) of the set $T_{j_k\cdot W'}$ have the form given by figure \ref{figure:forme}, 
\begin{figure}[h!]
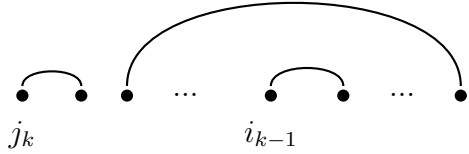

\begin{center}
\vspace{1cm}
\begin{psmatrix}[colsep=0.6,rowsep=0.1]
\pscircle*{0.08} & \pscircle*{0.08} & \pscircle*{0.08} & ... & 
\pscircle*{0.08} & \pscircle*{0.08} & ... & \pscircle*{0.08}\\
$j_k$ & & & & $i_{k-1}$ & & &
\end{psmatrix}
\ncarc[arcangle=90]{1,1}{1,2}
\ncarc[arcangle=90]{1,3}{1,8}
\ncarc[arcangle=90]{1,5}{1,6}
\end{center}
\caption{The two first blocks of the set $T_{j_k\cdot W}$.}
\label{figure:forme}
\end{figure}
where all reflections having in their supports an index in $[j_k+2, i_{k-1}-1]$ must have the other index of their support bigger than or equal to $i_{k-1}+2$ (otherwise $s_{i_{k-1}}$ would not be the lowest simple reflection in $T_{W'}$). Thanks to this property together with lemma \ref{lem:reflexions-recurrence} and the fact that $i_k<i_{k-1}$, applying $i_k\cdots (j_k+1)$ to $j_k\cdot W'$ does not change the support of the corresponding dense set and gives a set whose lowest simple reflection is $s_{i_k}$ (see figure \ref{figure:grossesref} for an illustration: in that case $n(k)=j_k$). 
\begin{figure}[h!]
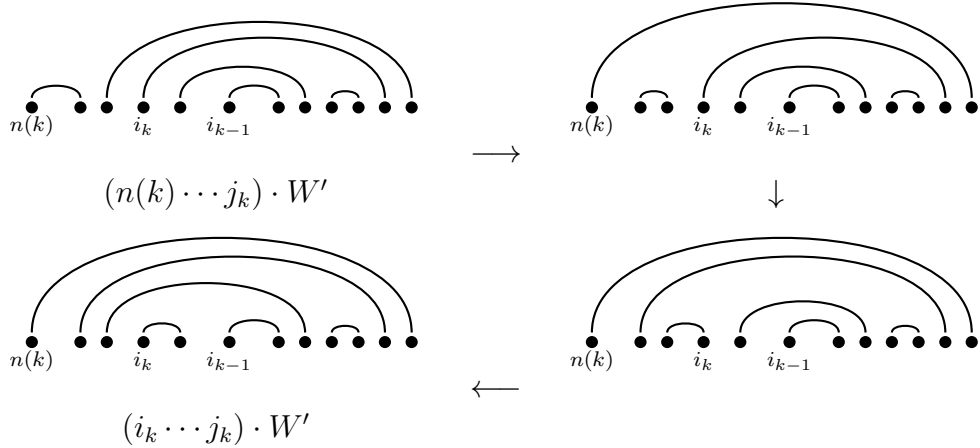

\begin{center}
\vspace{1.2cm}
\begin{tabular}{ccc}
\begin{tabular}{c}
\begin{psmatrix}[colsep=0.35,rowsep=0.1]
\pscircle*{0.08} & \pscircle*{0.08} & \pscircle*{0.08} & \pscircle*{0.08} & 
\pscircle*{0.08} & \pscircle*{0.08} & \pscircle*{0.08} & \pscircle*{0.08} & \pscircle*{0.08} & \pscircle*{0.08} & \pscircle*{0.08} & \pscircle*{0.08}\\
\scriptsize{$n(k)$} & & & \scriptsize{$i_k$} & & \scriptsize{$i_{k-1}$} & & &\\
\end{psmatrix}
\ncarc[arcangle=90]{1,1}{1,2}
\ncarc[arcangle=90]{1,3}{1,12}
\ncarc[arcangle=90]{1,4}{1,11}
\ncarc[arcangle=90]{1,5}{1,8}
\ncarc[arcangle=90]{1,6}{1,7}
\ncarc[arcangle=90]{1,9}{1,10}\\
\\
$(n(k)\cdots j_k)\cdot W'$\\
\end{tabular}
& $\longrightarrow$ & 
\begin{tabular}{c}
\begin{psmatrix}[colsep=0.35,rowsep=0.1]
\pscircle*{0.08} & \pscircle*{0.08} & \pscircle*{0.08} & \pscircle*{0.08} & 
\pscircle*{0.08} & \pscircle*{0.08} & \pscircle*{0.08} & \pscircle*{0.08} & \pscircle*{0.08} & \pscircle*{0.08} & \pscircle*{0.08} & \pscircle*{0.08}\\
\scriptsize{$n(k)$} & & & \scriptsize{$i_k$} & & \scriptsize{$i_{k-1}$} & & &\\
\end{psmatrix}
\ncarc[arcangle=90]{1,1}{1,12}
\ncarc[arcangle=90]{1,2}{1,3}
\ncarc[arcangle=90]{1,4}{1,11}
\ncarc[arcangle=90]{1,5}{1,8}
\ncarc[arcangle=90]{1,6}{1,7}
\ncarc[arcangle=90]{1,9}{1,10}\\
$~$\\
$\downarrow$\\
\end{tabular}\\
& & \\
& & \\
& & \\
\begin{tabular}{c}
\begin{psmatrix}[colsep=0.35,rowsep=0.1]
\pscircle*{0.08} & \pscircle*{0.08} & \pscircle*{0.08} & \pscircle*{0.08} & 
\pscircle*{0.08} & \pscircle*{0.08} & \pscircle*{0.08} & \pscircle*{0.08} & \pscircle*{0.08} & \pscircle*{0.08} & \pscircle*{0.08} & \pscircle*{0.08}\\
\scriptsize{$n(k)$} & & & \scriptsize{$i_k$} & & \scriptsize{$i_{k-1}$} & & &\\
\end{psmatrix}
\ncarc[arcangle=90]{1,1}{1,12}
\ncarc[arcangle=90]{1,2}{1,11}
\ncarc[arcangle=90]{1,4}{1,5}
\ncarc[arcangle=90]{1,3}{1,8}
\ncarc[arcangle=90]{1,6}{1,7}
\ncarc[arcangle=90]{1,9}{1,10}\\
$~$\\
$(i_k\cdots j_k)\cdot W'$\\
\end{tabular}
& $\longleftarrow$ & 
\begin{tabular}{c}
\begin{psmatrix}[colsep=0.35,rowsep=0.1]
\pscircle*{0.08} & \pscircle*{0.08} & \pscircle*{0.08} & \pscircle*{0.08} & 
\pscircle*{0.08} & \pscircle*{0.08} & \pscircle*{0.08} & \pscircle*{0.08} & \pscircle*{0.08} & \pscircle*{0.08} & \pscircle*{0.08} & \pscircle*{0.08}\\
\scriptsize{$n(k)$} & & & \scriptsize{$i_k$} & & \scriptsize{$i_{k-1}$} & & &\\
\end{psmatrix}
\ncarc[arcangle=90]{1,1}{1,12}
\ncarc[arcangle=90]{1,2}{1,11}
\ncarc[arcangle=90]{1,3}{1,4}
\ncarc[arcangle=90]{1,5}{1,8}
\ncarc[arcangle=90]{1,6}{1,7}
\ncarc[arcangle=90]{1,9}{1,10}\\
$~$\\
$~$\\
\end{tabular}
\end{tabular}
\caption{Example of the process of applying the sequence $i_k\cdots (n(k)+1)$ to $(n(k)\cdots j_k)\cdot W'$; in case $|S_k|=1$ we have $n(k)=j_k$.}
\label{figure:grossesref}
\end{center}
\end{figure}

Now suppose $|S_k|=2$ ; applying the sequence $n(k)\cdots(j_k+1)j_k$ to $W'$ we get a variety $W''$ with corresponding set equal to $T_{W'}\cup\{s_{n(k)}\}$ since $n(k-1)=n(k)+2$ is the lowest index in $T_{W'}$. We can then argue exactly as in the first case to get the conclusion (see figure \ref{figure:grossesref}).   
\end{proof}
\begin{prop}\label{prop:gi}
Let $B$ be a generalized intertwined bimodule with associated sequence $i_k\cdots j_k\cdots i_1\cdots j_1$. Then $B$ is indecomposable. More precisely, when writing $B$ in the form 
$$B(k)\otimes_{R_{j_k}} \mathcal{O}(W_{k})\otimes_{\mathcal{O}(X_k)} B(k-1)\otimes\cdots \otimes B(2)\otimes_{R_{j_2}}\mathcal{O}(W_2)\otimes_{R_{i_1}} B(1)$$
where we made the same choice of brackets as in proposition \ref{prop:entrelace}, with $X_\ell$ the variety associated to the subsequence $i_{\ell-1}\cdots j_{\ell-1}\cdots i_1\cdots j_1$ and $W_\ell=X_\ell\cap V_{j_\ell}$, any tensor in $B$ can be written as a sum of elements of the form
$$a\cdot 1\otimes 1\otimes \cdots \otimes 1\otimes 1\cdot p(f_1,\dots, f_{i_1})$$
\noindent where the $\cdot$ holds for the operation of $\bar{R}$ on both sides and $p(f_1,\dots, f_{i_1})$ is a polynomial in $f_1,f_2,\dots, f_{i_1}$.

Moreover if $j+2$ is smaller than or equal to the smallest index in $S_k$, then there exists a polynomial $p(f_1,\dots, f_{i_1})$ such that
$$f_j\cdot 1\otimes 1\otimes\cdots 1\otimes 1=1\otimes 1\otimes\cdots 1\otimes 1\cdot p(f_1,\cdots f_{i_1}).$$
\end{prop}
\begin{proof}
We first consider in which case the variety $X_\ell$ is $s_{j_{\ell}}$-invariant ; if $|S_{\ell}|=2$, we have that $j_\ell\leq n(\ell-1)-2$ by definition \ref{def:gi} hence $X_\ell$ is $s_{j_\ell}$-invariant by the first assertion of lemma \ref{lem:indiceminimal} together with proposition \ref{prop:reflexion}. If $|S_\ell|=1$, then $j_{\ell}=n(\ell-1)-1$ by definition \ref{def:gi} hence $X_{\ell}$ is not $s_{j_\ell}$-invariant by the first assertion of lemma \ref{lem:indiceminimal} together with proposition \ref{prop:reflexion}. Therefore in case $|S_\ell|=2$ one can decompose
$$\mathcal{O}(W_{\ell})=\mathcal{O}(W_{\ell})^{s_{j_{\ell}} }\oplus \mathcal{O}(W_{\ell})^{s_{j_{\ell}} } f_{j_{\ell}}|_{W_{\ell}},$$ hence for each $\ell$ such that $|S_\ell|=2$ we can decompose the $\mathcal{O}(W_\ell)$-component of any tensor in $B$ and move the invariant parts to the left in $B(\ell)$ and then in $\mathcal{O}(W_{\ell+1})$ as we did in \ref{prop:entrelace} for the interwtined case. 
In the case where $|S_\ell|=1$, we have seen that $X_\ell$ is not $s_{j_{\ell}}$-invariant. Thanks to corollary \ref{cor:ni} together with remark \ref{rmq:noninvariant}, $R_{j_{\ell}}^{s_{j_{\ell}}}\twoheadrightarrow \mathcal{O}(W_\ell)$, hence the $\mathcal{O}(W_\ell)$ component of any tensor in $B$ can be moved to the left in $B(\ell)$ and then in the $\mathcal{O}(W_{\ell+1})$-component. As a consequence, any tensor $b\in B$ can be written as a sum $\sum_i a_i\cdot t_i$, where $a_i\in\bar{R}$ and $t_i$ are tensors in $B$ with $f_{j_{\ell}}$ or $1$ in the $\mathcal{O}(W_{\ell})$-component for $\ell$ such that $|S_\ell|=2$, with $1$ in $\mathcal{O}(W_{\ell})$-component for $\ell$ such that $|S_\ell|=1$ and with $1$ in the components coming from the bimodules $B(\ell)$. It remains to show that  if $|S_{\ell}|=2$, the $f_{j_{\ell}}$ in the $\mathcal{O}(W_{\ell})$-components can be "moved to the right". 

Now we consider an element $f_j$ in the $\mathcal{O}(W_\ell)$-component of one of the $t_i$, with $j\leq n(\ell)$ as we did at the end of the proof of \ref{lem:tech}. If $|S_{\ell-1}|=1$, then the only index in $S_{\ell-1}$ is $j_{\ell-1}$ and one has $j_{\ell-1}\geq j+2$ since $|S_{\ell}|=2$. In that case, any index occuring in the sequence $i_{\ell-1}\cdots j_{\ell-1}\cdots i_1\cdots j_1$ is distant from $j$ and hence $f_j$ can be moved in the very first component on the right of our tensor product (that is $\mathcal{O}(W_1)=R_{j_1}$). The other case is the case where $|S_{\ell-1}|=2$. Since $i_{\ell-1}>n(\ell)+1$, $f_j$ is $s_{i_{\ell-1}}$-invariant and hence can be moved to the right in $B(\ell-1)$. We then argue exactly as in lemma \ref{lem:tech}, distinguishing the three cases: $j<j_{\ell-1}-1$, $j=j_{\ell-1}-1$ and $j\geq j_{\ell-1}$, to conclude that we can "move" our $f_j$ to the right in the $\mathcal{O}(W_{\ell-1})$-component where we obtain a sum of $f_{j'}$ for $j'\leq j+2$. But since $|S_{\ell}|=2$, $j'\leq n(\ell-1)$. Hence we can inductively "move" the $f_j$'s to the $\mathcal{O}(W_m)$-component with $m<\ell$ as far as $|S_i|=2$ for each $i\in[m, \ell-1]$  obtaining in that component a polynomial in $p(f_1,\cdots, f_{n(m)})$ and if then $|S(m-1)|=1$, we apply the first case to move our polynomial in the very first component on the right of the tensor product (that is $\mathcal{O}(W_1)=R_{j_1}$). Hence we can inductively move any $f_j$ to the right and one obtains in that component polynomials in the $f_i$'s for $i$ smaller than or equal to $n(1)=i_1$. This also shows the last statement since if $j+2$ is less than or equal to $n(k)$, then arguing as above our $f_j$ lying in the very first component on the left of the tensor product can be moved in the $\mathcal{O}(W_k)$-component and one obtains a sum of $f_{j'}$ for $j'$ less than or equal to $j+2\leq n(k)$. 
\end{proof}
\noindent We have all the required tools to prove :
\begin{thm}\label{thm:youpi}
Let $B$ be a fully commutative bimodule. Then $B$ is indecomposable in $\bar{R}-\mathrm{mod}_{\mathbb{Z}}-\bar{R}$. 
\end{thm}
\begin{proof}
We consider the sequence $i_k\cdots j_k\cdots i_1\cdots j_1$ our bimodule is associated to. We consider the biggest index $\ell$ such that the bimodule associated to the subsequence 
$\mathrm{seq_1}=i_\ell\cdots j_\ell\cdots i_1\cdots j_1$ is a generalized intertwined bimodule and write $G(1)$ for the corresponding bimodule. Then one can do the same with the subsequence $i_k\cdots j_k\cdots i_{\ell+1}\cdots j_{\ell+1}$ to obtain a generalized intertwined bimodule $G(2)$ associated to a subsequence $\mathrm{seq_2}$. At the end of the process we obtain a sequence $G(1),\dots, G(m)$ of intertwined bimodules associated to subsequences $\mathrm{seq}_1,\dots,\mathrm{seq}_m$ such that 
$$B\cong G(m)\ast G(m-1)\ast\cdots\ast G(2)\ast G(1)$$
and $\mathrm{seq}=\mathrm{seq}_m\cdots\mathrm{seq}_2\mathrm{seq}_1$. We compute the various $\ast$ products occuring in each of the bimodule $G(i)$ with the same choice of brackets as in propositions \ref{prop:entrelace} and \ref{prop:gi}; we then compute the above product "from the right", i.e. with the following choice of brackets :
$$G(m)\ast (G(m-1)\ast(\cdots (G(3)\ast (G(2)\ast G(1)))\cdots)).$$
By maximality of the rank of the subsequence $i_\ell\cdots j_\ell\cdots i_1\cdots j_1$ defining $G(1)$, $j_{\ell+1}\leq i_{\ell+1}<n(\ell)-1$. But we know from lemma \ref{lem:indiceminimal} that the lowest index in the support of $T_{U_2}$ where $U_2$ is the variety associated to $\mathrm{seq}_1$ is precisely $n(\ell)$. The variety $Z_2$ occuring when computing the $\ast$ product between $G(1)$ and $G(2)$, which is equal to $U_2\cap V_{j_{\ell+1}}$, is then $s_{j_{\ell+1}}$-invariant. Moreover, since $i_{\ell+1}$ is the biggest index occuring in $\mathrm{seq}_2$, one has that
$$T_{W_{\mathrm{seq}_2\mathrm{seq}_1}}=T_{W_{\mathrm{seq}_1}}\cup T_{W_{\mathrm{seq}_2}}$$
and the same holds using induction when replacing $1$ by $i$ for $1<i<m$. Hence our bimodule is isomorphic to 
$$G(m)\otimes_{R_{k_m}} \mathcal{O}(Z_m)\otimes_{\mathcal{O}(U_m)} G(m-1)\otimes\cdots\otimes_{R_{k_2}} \mathcal{O}(Z_2)\otimes_{\mathcal{O}(U_2)} G(1)$$
where $U_j$ is the variety associated to the sequence $\mathrm{seq}_{j-1}\cdots\mathrm{seq}_2 \mathrm{seq}_1$, $k_j$ is the last index of the sequence $\mathrm{seq}_j$ and $Z_j=U_j\cap V_{k_j}$ is $s_{k_j}$-invariant. Now consider any tensor 
$$a_m\otimes_{R_{k_m}} b_m\otimes_{\mathcal{O}(U_m)} a_{m-1}\otimes\cdots\otimes_{R_{k_2}} b_2\otimes_{\mathcal{O}(U_2)} a_1$$
\noindent in the above tensor product with $a_j\in G(j)$, $b_j\in\mathcal{O}(Z_j)$. Since $R_{k_j}\twoheadrightarrow \mathcal{O}(Z_j)$ we can suppose that each $b_i$ equals $1$. Now using proposition \ref{prop:gi} inductively, beginning with $a_1$, we can rewrite our tensor as a sum of tensors of the form 
$$a\cdot 1\otimes_{R_{k_m}} p(f_1,\dots, f_{n_m})\otimes_{\mathcal{O}(U_m)} 1\otimes\cdots\otimes_{R_{k_2}} p(f_1,\dots,f_{n_2})\otimes_{\mathcal{O}(U_2)} 1\cdot p(f_1,\cdots,f_{n_1}),$$
where $n_j$ is the biggest index in the sequence $\mathrm{seq}_j$ (in particular $n_1=i_1$ and $n_2=i_{\ell+1}$. Now each $n_j+2$ is less than or equal to the smallest index in the set of intertwining indices of $\mathrm{seq}_{j-1}$ because this sequence was chosen to be maximal such that the corresponding bimodule is a generelized intertwining bimodule. Hence we can apply the last statement of proposition \ref{prop:gi} inductively, beginning from the left. This concludes
\end{proof}

\subsection{Categorification of the Kazhdan-Lusztig basis}

Notice that the category of finitely generated graded $\bar{R}$-bimodules has the Krull-Schmidt property (see \cite{S}, remark 1.3). Thanks to theorem \ref{thm:youpi}, we can extend the $\ast$ product to direct sums of fully commutative bimodules and their graded shifts by bilinearity. 

\begin{nota}
We write $\mathcal{B}_{\mathrm{TL}_n}$ for the additive monoidal category generated by $\ast$-products of fully commutative bimodules and their shifts and stable by direct sums (and direct summands, but an indecomposable direct summand of a product of shifts of fully commutative bimodules is again a shift of a fully commutative bimodule). Recall that for $w\in\mathcal{W}_c$ a fully commutative element, we write $b_w$ for the corresponding element of the Temperley-Lieb algebra and $B_w$ for the corresponding fully commutative bimodule. 
\end{nota}
\noindent Combining our efforts from the previous sections we get

\begin{thm}[\textbf{Categorification of the Kazhdan-Lusztig basis of the Temperley-Lieb algebra}]\label{thm:fin}
The category $\mathcal{B}_{\mathrm{TL}_n}$ categorifies the Kazhdan-Lusztig basis of the Temperley-Lieb algebra $\mathrm{TL}_n$. More precisely, we have an isomorphism $\mathcal{E}$ of $\mathbb{Z}[\tau, \tau^{-1}]$-algebras
$$
\mathrm{TL}_n\overset{\sim}{\longrightarrow} \left\langle \mathcal{B}_{\mathrm{TL}_n},\oplus,\ast\right\rangle,\\
$$
such that $\mathcal{E}(b_w)=B_w$, $\mathcal{E}[\tau]=\bar{R}[1]$. Here $\left\langle \cdot\right\rangle$ stands for the split Grothendieck group of the additive category (which becomes a ring with $\ast$).
\end{thm}
\begin{proof}
We know from theorem \ref{thm:relations} that the bimodules $B_w$ satisfy the Temperley-Lieb relations. This shows that we have a surjective morphism of $\mathbb{Z}[\tau, \tau^{-1}]$-algebras. In order to see that this morphism is injective, it suffices to show that if $w\neq w'$ are two fully commutative elements in $\mathcal{W}$, then the corresponding bimodules $B_{w}$ and $B_{w'}$ are nonisomorphic. This has already been proven in \ref{cor:noniso}.
\end{proof}


\begin{thebibliography}{4}
\bibitem{BE} B. Elias, \textsl{A diagrammatic Temperley-Lieb categorification}, International Journal of Mathematics and Mathematical Sciences, Vol. 2010 (2010). 

\bibitem{Carter} R. W. Carter, \textsl{Conjugacy classes in the Weyl group}, Compositio Mathematica, Vol. 25, Fasc. 1, 1972, 1-59. 

\bibitem{FG} C.K. Fan, R.M. Green, \textsl{Monomials and Temperley-Lieb algebras}, Journal of Algebra 190, 1997, 498-517. 

\bibitem{GL} R.M. Green, J. Losonczy, \textsl{Canonical bases for Hecke algebra quotients}, Math. Research Letters 6, 1999, 213-222.

\bibitem{KT} C. Kassel, V. Turaev, \textsl{Braid groups}, Graduate Texts in Mathematics, Vol. 247, Springer, 2008. 

\bibitem{KL} D. Kazhdan, G. Lusztig, \textsl{Representations of Coxeter Groups and Hecke Algebras}, Inventiones math., 1979, 165-184.

\bibitem{Soergel} W. Soergel, \textsl{Kategorie $\mathcal{O}$, Perverse Garben und Moduln über den Koinvarianten zur Weylgruppe}, Journal of the American Mathematical Society, Vol. 3, No. 2, 1990.

\bibitem{S} W. Soergel, \textsl{Kazhdan-Lusztig polynomials and indecomposable bimodules over polynomial rings}, Journal of the Inst. Math. Jussieu 6(3), 2007, 501-525.

\bibitem{Stroppel} C. Stroppel, \textsl{Category $\mathcal{O}$ : gradings and translation functors}, Journal of Algebra 268, 2003, 301-326.

\bibitem{Strop} C. Stroppel, \textsl{Categorification of the Temperley-Lieb category, tangles, and cobordisms via projective functors}, Duke Mathematical Journal, Vol. 126, No. 3, 2005, 547-596.
\end{thebibliography}
\end{document}